\documentclass[10pt]{paper}%
\usepackage[T1]{fontenc}
\usepackage[latin1]{inputenc}
\usepackage{makeidx}
\usepackage{amsfonts}
\usepackage{amssymb}
\usepackage{float}
\usepackage{amsmath}
  \usepackage[all]{xy}
  \usepackage{graphicx}
  \usepackage{youngtab}
  \usepackage{amsthm}
  \usepackage{nicefrac}
  \usepackage{fullpage}
  \usepackage{pstricks}
  \usepackage{color}
  \usepackage{mathtools}
  \usepackage{enumerate}

  \usepackage{color}
  
\newtheorem*{Th}{Theorem}
\newtheorem*{Lem}{Lemma}
\newtheorem*{Cor}{Corollary}
\newtheorem*{prop}{Proposition}

\theoremstyle{remark}
\newtheorem*{Rem}{Remark}{\rmfamily}
\theoremstyle{definition}
{\rmfamily}
\newtheorem*{Exa}{Example}{\rmfamily}

  \newtheorem{abs}[subsection]{\bfseries}

\newcommand{\nd}{\noindent}
\newcommand{\ds}{\displaystyle}
\renewcommand{\lq}{\leqslant}
\newcommand{\gq}{\geqslant}

\newcommand{\p}{\mathbb{P}}
\newcommand{\G}{\mathbf{G}}

\newcommand{\Z}{\mathbb{Z}}
\newcommand{\R}{\mathbb{R}}
\newcommand{\C}{\mathbb{C}}
\newcommand{\Q}{\mathbb{Q}}
\newcommand{\X}{\mathbf{X}}
\newcommand{\fc}{\mathbf{c}}
\newcommand{\Y}{\mathbf{Y}}

\newcommand{\h}{\mathbf{h}}
\renewcommand{\l}{\ell}

\renewcommand{\H}{\mathbf{H}}
\newcommand{\m}{\mathbf{m}}
\renewcommand{\a}{\mathbf{a}}
\newcommand{\s}{\mathbf{s}}
\newcommand{\z}{\mathbf{z}}
\newcommand{\q}{\mathbf{q}}
\newcommand{\M}{\mathcal{M}}

\newcommand{\ttheta}{%
\setlength{\unitlength}{1pt}
    \begin{picture}(6.5,0)%
      \put(-.85,0){$\theta$}%
      \put(-1.1,0){$\theta$}%
    \end{picture}%
} 

\newcommand{\tthetatilde}{%
\setlength{\unitlength}{1pt}
    \begin{picture}(6.5,0)%
      \put(-.85,0){$\widetilde \theta$}%
      \put(-1.1,0){$\widetilde \theta$}%
    \end{picture}%
} 

\newcommand{\lgr}{%
\setlength{\unitlength}{1pt}
    \begin{picture}(6.5,0)%
      \put(-.85,0){$\l$}%
      \put(-1.1,0){$\l$}%
    \end{picture}%
} 

\newcommand{\tthetabar}{%
\setlength{\unitlength}{1pt}
    \begin{picture}(6.5,0)%
      \put(-.85,0){$\bar\theta$}%
      \put(-1.1,0){$\bar\theta$}%
    \end{picture}%
}

\newcommand{\is}{\mbox{\scriptsize\boldmath${s}$}}

\DeclareMathOperator{\ho}{Hom}
\DeclareMathOperator{\res}{Res}

\DeclareMathOperator{\gl}{GL}
\DeclareMathOperator{\proj}{Proj}

\DeclareMathOperator{\lie}{Lie}

\DeclareMathOperator{\mat}{Mat}

\DeclareMathOperator{\spec}{Specm}

\DeclareMathOperator{\irr}{Irr}

\renewcommand{\a}{%
\setlength{\unitlength}{1pt}
    \begin{picture}(6.5,0)%
      \put(-.85,0){$\mbi{a}$}%
      \put(-1.1,0){$\mbi{a}$}%
    \end{picture}%
}

\newcommand{\llambda}{%
\setlength{\unitlength}{1pt}
    \begin{picture}(6.5,0)%
      \put(-.85,0){$\lambda$}%
      \put(-1.1,0){$\lambda$}%
    \end{picture}%
} 	

\newcommand{\mmu}{%
\setlength{\unitlength}{1pt}
    \begin{picture}(6.5,0)%
      \put(-.85,0){$\mu$}%
      \put(-1.1,0){$\mu$}%
    \end{picture}%
} 

\newcommand{\llambdabar}{%
\setlength{\unitlength}{1pt}
    \begin{picture}(6.5,0)%
      \put(-.85,0){$\overline\lambda$}%
      \put(-1.1,0){$\overline\lambda$}%
    \end{picture}%
} 	

\newcommand{\mmubar}{%
\setlength{\unitlength}{1pt}
    \begin{picture}(6.5,0)%
      \put(-.85,0){$\overline\mu$}%
      \put(-1.1,0){$\overline\mu$}%
    \end{picture}%
}

\definecolor{vert}{rgb}{0,0.7,0.3}
\definecolor{gris25}{gray}{0.85}

\def\cat#1{{\mathfrak{#1}}}

\author{{Emilie Liboz} }
\title{\textbf{Orderings on Calogero-Moser partition of imprimitive groups}}

\date{}

\makeindex

\DeclareMathAlphabet\mbi{OML}{cmm}{b}{it}

\begin{document}

\renewcommand{\labelitemi}{$\bullet$}
\maketitle


\nd We extend to all parameters the constructions of the geometric and combinatorial orders on $\irr G(\l,1,n)$ due to \cite{go:qu:08}, as well as the relations with the $\a$ and $\fc$-functions. This allows us to generalize these properties for the group $G(\l,e,n)$, at least for $e \nmid n$.


\section{Introduction}

This paper deals with the combinatorial representation theory, in particular the description of so-called \textit{families} for the complex reflection groups $G(\l,e,n)$ as well as the partial ordering that exists on this partition. This field is intimately related to the geometry of flag varieties and the Lie representation theory.

\nd There has been recent partial progress on this using the geometry of quiver varieties and the representation theory of rational Cherednik algebras : let $W$ be a finite complex reflection group, the blocks of $\overline{H_c(W)}$, the restricted rational Cherednik algebra at $t=0$ attached to $W$, induce a partitioning $CM_c(W)$ of the set $\irr W$ called  the \textit{Calogero-Moser partition}. This could generalize the notion of families if $W$ is not a Coxeter group (see \cite{go:qu:08}, \cite{go:ca:09}, \cite{ma:ca:10}  and \cite{be:th:09}).\\

\nd The aim of this paper is to study certain natural orderings on the set $CM_c(G(\l,e,n))$ constructed numerically (by $\a$ or $\fc$-functions), combinatorially (throught the combinatorics of $\l$-cores and quotients) and geometrically (via the Bialynicki-Birula decomposition).

\nd In the case when $e$ does not divide $n$, we give explicit links between the geometric and the combinatorial orderings and the $\a$ and $\fc$-functions; moreover for $\l=2$, we relate the geometric and the combinatorial orderings. 

\nd In the case when $e$ divides $n$, we propose an interesting normal variety that could play the role that quiver varieties play for $G(\l,1,n)$, i.e. a variety $\M$ such that: \begin{itemize} 
 \item there exists a $\C^*$-equivariant morphism  $\M\twoheadrightarrow (\cat{h}\times\cat{h}^*)/G(\l,e,n)$,
 \item $\M^{\C^*}$ is in one-to-one correspondence with $CM_c(G(\l,e,n))$,
 \item the geometric order defined on $CM_c(G(\l,e,n))$ using the Bialynicki-Birula decomposition of $\M$ is related to the other orderings
 \end{itemize} and give evidences for this in the case $\l=e=n=2$.



\section{Ordering on $\C^*$-fixed points of a normal variety}

In this section, we build an ordering on the set of $\C^*$-fixed points of a normal variety, in order to describe geometrically some combinatorial or algebraic order relations defined for some complex reflexion groups.

\begin{abs}\label{para1} 

Let $X$ be a quasiprojective complex variety with a regular action of $\C^*$ such that $|X^{\C^*}|<\infty$. For $x_0\in X^{\C^*}$, we can define the attracting set $$X_{x_0}=\{x\in X,~\lim_{\eta\rightarrow 0}\eta\cdot x=x_0\}.$$

\nd The transitive closure of the rule $$x\prec x' ~\Longleftrightarrow ~\overline{X_x}\cap X_{x'}\neq \emptyset$$ defines a preorder relation on $X^{\C^*}$, where $\overline{X_x}$ is the Zariski-closure of $X_x$.

\begin{Rem} In general this relation is not antisymmetric, but if $X$ is smooth then the Bialynicki-Birula decomposition of $X$ $$X=\bigsqcup_{x\in X^{\C^*}} X_x$$ is filtrable (see \cite{bi:so:76}) and $\prec$ is an order relation.
\end{Rem} We will now show a more general result.

\begin{Th}
If $X$ is a normal variety then the relation $\prec$ defines an order relation on $X^{\C^*}$.
\end{Th}

\begin{proof}
If $X$ is a normal variety then, by \cite[Th. 1]{su:eq:74}, there is a $\C^*$-equivariant locally closed embedding $\iota~:~X\hookrightarrow \p^N(\C)$ for some $N$, where $\p^N(\C)$ has a $\C^*$-action of the form $\eta\cdot(x_0:\cdots:x_N)=(\eta^{w_0}x_0:\cdots:\eta^{w_N}x_N)$ with $w_i\in\Z$, for any $\eta\in\C^*$ and $(x_0:\cdots:x_N)\in\p^N(\C)$. 

\nd Let us describe the attracting sets of $\p^N(\C)$. We can write $\{0,\dots,N\}=C_1\sqcup \dots\sqcup C_r$, where $C_i=\{k\in\{0,\dots,N\},~w_k=c_i\}$ and $c_1<c_2<\dots<c_r$. Then the fixed points of $\p^N(\C)$ are $(x_0:\cdots :x_N)\in\p^N(\C)$ where $\forall j\notin C_i,~x_j=0$. Thus the connected components of $\p^N(\C)^{\C^*}$ are $$W_i=\{(x_0:\cdots :x_N)\in\p^N(\C), \forall j\notin C_i,~x_j=0\}$$ for $1\lq i\lq r$ and $W_i$ is isomorphic to $\p^{|C_i|-1}(\C)$.
Then the set
 $$\begin{array}{rcl}\p^N(\C)_i&:=&\bigcup_{x\in W_i}\p^N(\C)_x\\&=&\{(y_0:\cdots :y_n)\in\p^N(\C),

\exists j\in C_i,~y_j\neq 0\text{ and }\forall j\in C_1\sqcup\dots \sqcup C_{i-1},~y_j=0\}\end{array}$$ has the following properties : $$\p^N(\C)_i\simeq\p^{|C_i|-1}(\C)\times\mathbb{A}^{|C_{i+1}|+\cdots+|C_r|}(\C),~\overline{\p^N(\C)_i}\simeq\p^{|C_{i}|+\cdots+|C_r|-1}(\C)\text{ and }\displaystyle{\overline{\p^N(\C)_i}=\bigcup_{k\geqslant i}\p^N(\C)_k}.$$ Moreover, if $x\in W_i$ then $\p^N(\C)_x=p^{-1}\{x\}$ is closed in $\p^N(\C)_i$, where $p$ is the projection on $\p^{|C_i|-1}(\C)$.

\nd Let $x$ et ${x'}\in X^{\C^*}$ such that $\overline{X_x}\cap X_{x'}\neq \emptyset$ et $\overline{X_{x'}}\cap X_x\neq \emptyset$. We have to show that $x={x'}$, using the embedding $\iota~:~X\hookrightarrow \p^N(\C)$. If $\iota(x)\in W_i$ then $X_x$ satisfies $\iota(\overline{X_x})\subset \overline{\iota(X_x)}$ and $\iota(X_x)\subset \p^N(\C)_{\iota(x)}\subset \p^N(\C)_i$. Thus, if $\iota(x)\in W_i$ and $\iota({x'})\in W_j$, then combining $\overline{X_x}\cap X_{x'}\neq \emptyset$ with $\overline{X_{x'}}\cap X_x\neq \emptyset$, we obtain $\overline{\p^N(\C)_i}\cap \p^N(\C)_j\neq \emptyset$ and $\overline{\p^N(\C)_j}\cap \p^N(\C)_i\neq \emptyset$. Thus, it follows from the above description of $\overline{\p^N(\C)_i}$ that $i=j$. Finally, since $ \p^N(\C)_{\iota(x)}$ and $ \p^N(\C)_{\iota({x'})}$ are closed on $\p^N(\C)_i$, the fact that $\overline{\p^N(\C)_{\iota(x)}}\cap\p^N(\C)_{\iota({x'})}$ is not empty shows that $\iota(x)=\iota({x'})$ and $x={x'}$.

\end{proof}

\end{abs}

\begin{abs}\label{2.2}

Let $X$ and $Y$ be two varieties as in \S \ref{para1}. Assume that there exists a $\C^*$-equivariant, surjective and projective morphism  $\pi~:~X\rightarrow Y$. Then, providing $X=\bigsqcup_{x\in X^{\C^*}}X_x$, we will compare the attracting sets of $X$ and $Y$. For instance, it is easy to check that:

\begin{Lem}\label{lem1}
 $\pi(X^{\C^*})=Y^{\C^*}$
and for $y\in Y^{\C^*}$, $Y_{y}=\displaystyle{\bigcup_{\substack{x\in X^{\C^*} \\ \pi(x)=y}}\pi(X_x)}$
\end{Lem}

\nd If $X$ and $Y$ are normal we will denote by $\prec$ the two orders on $X^{\C^*}$ and $Y^{\C^*}$ defined as in \ref{para1}. The compatibility of $\pi$ with these orders is established in our next result.

\begin{prop}\label{theo1}
Let $x,~x'\in X^{\C^*}$. \begin{enumerate} \item $x\prec x'~\Rightarrow~\pi(x)\prec\pi(x')$
\item $\begin{array}{rcl}\pi(x)\prec\pi(x') & \Rightarrow & \exists x_1,x_2,x_2',x_3',\dots,x_m ~:~\pi(x_1)=\pi(x),~\pi(x_m)=\pi(x'),\\ & & \forall~2\lq i\lq m-1,~\pi(x_i)=\pi(x_i') \\& & \forall~1\lq i\lq m-1,~x_i\prec x_{i+1}\text{ if } i\not\in 2\Z\text{ and } x_i'\prec x_{i+1}' \text{ if }i\in 2\Z. \end{array}$
\end{enumerate}
\end{prop}

\begin{proof} 

\begin{enumerate}

\item We use Lemma \ref{lem1} and the fact that $\pi$ is closed to show that $$\overline{X_x}\cap X_{x'}\neq \emptyset ~\Rightarrow ~\emptyset\neq \pi(\overline{X_x}\cap X_{x'})\subset\pi(\overline{X_x})\cap \pi(X_{x'})=\overline{\pi(X_x)}\cap \pi(X_{x'})\subset \overline{Y_{\pi(x)}}\cap Y_{\pi(x')}$$ and we conclude by definition of $\prec$.

\item In the same manner, since  $X$ has a finite number of fixed points, it is easily seen that, for $y,~y'\in Y^{\C^*}:$ $$\overline{Y_{y}}\cap Y_{y'}\neq \emptyset~\Rightarrow~\emptyset\neq\displaystyle{\left(\overline{\bigcup_{\substack{x\in X^{\C^*} \\ \pi(x)=y}}\pi(X_x)}\right)\bigcap\left(\bigcup_{\substack{x'\in X^{\C^*} \\ \pi(x')=y'}}\pi(X_{x'})\right)}= \displaystyle{ \left(\bigcup_{\substack{x\in X^{\C^*} \\ \pi(x)=y}}\pi(\overline{X_x})\right)\bigcap\left(\bigcup_{\substack{x'\in X^{\C^*} \\ \pi(x')=y'}}\pi(X_{x'})\right)}.$$ We can now show, using the equality $X=\bigsqcup_{x\in X^{\C^*}}X_x$, that there exists $x,z\in X^{\C^*}$ such that $\pi(x)=y$, $\pi(z)=y'$ and $\overline{X_x}\cap X_z\neq \emptyset$ and we conclude by definition of $\prec$.

\end{enumerate}

\end{proof}

We suppose now that $\pi$ is the quotient by a finite group $G$ that acts regularly on $X$ and such that this action commutes with the $\C^*$-action. In this case, $Y=X/G$ and all the points of $\pi^{-1}(y)$, with $y\in (X/G)^{\C^*}$, are fixed. We thus get a bijection between $X^{\C^*}/G$ and $(X/G)^{\C^*}$. Moreover, using the fact that $g\cdot X_x=X_{g\cdot x}$ for $g\in G$, we can improve the result of Proposition \ref{theo1}.

\begin{Cor}
If $Y=X/G$ where $G$ is a finite group, then for $x,~x'\in X^{\C^*}$: $$\pi(x)\prec\pi(x')~\Longleftrightarrow~\exists~g\in G\text{ such that }x\prec gx'.$$
\end{Cor}

\end{abs}


\section{The cast}

\begin{abs}{\bf Groups $\mathbf{G(\lgr,1,n)}$.} \hspace{0,05cm} A partition of a positive integer $n$ is a nonincreasing sequence $\lambda=(\lambda_1\gq \lambda_2\gq\cdots\gq\lambda_r)$ of positive integers with sum $n$. We write $|\lambda|=n$ and we denote the dominance order on the set $\mathcal{P}(n)$ of partitions of $n$ by $\unlhd$, so that $\lambda \unlhd\mu$ when $\lambda_1+\cdots+\lambda_i\lq\mu_1+\cdots+\mu_i$ for all $i$. An $\l$-multipartition of $n$ is an $\l$-tuple of partitions $\llambda=(\lambda^1,\dots,\lambda^\l)$ with $\ds{\sum_{i=1}^\l |\lambda^i|=n}$.

 We first concentrate on the complex reflection group $W=G(\l,1,n)\simeq (\mu_\l)^n\rtimes \cat{S}_n$ (see the classification of Shephard-Todd in \cite{sh:fi:54}), where $\l$ and $n$ are natural numbers, $\mu_\l$ is the cyclic group of order $\l$ and $\cat{S}_n$ is the symmetric group on $n$ elements. This group acts naturally on its reflection representation $\cat{h}=\C^n$ and $\irr W$ is labeled by the set $\mathcal{P}(\l,n)$ of $\l$-multipartitions of $n$: $$\irr W=\{E_\lambda,~\llambda=(\lambda^1,\dots,\lambda^\l)\in\mathcal{P}(\l,n)\}.$$ The set $\mathcal{S}_W$ of complex reflections of $W$ has $\l$ conjugacy classes: $ \mathcal{S}_0=\{ s_{i,j}\sigma_i^r\sigma_j^{-r}, ~1\leqslant i\neq j\leqslant n \text{ and } 0\leqslant r\leqslant \l-1\}$ 
and $\mathcal{S}_t=\{\sigma_i^t,~1\leqslant i\leqslant n\}$, for $1\leqslant t\leqslant \l-1$, where $s_{i,j}$ is the permutation matrix and $\sigma_i$ is the diagonal matrix with just one term $a_{i,i}$ different from $1$ which is equal to $\zeta_\l=\exp\left(\frac{2\pi\sqrt{-1}}{\l}\right)$. 
\end{abs}

\begin{abs}{\bf Cherednik algebras.} \hspace{0,05cm} 
Let $\h=(h,H_1,\dots,H_{\l-1})\in\mathbb Q^\l$, $(y_1,\dots,y_n)$ the canonical basis of $\cat{h}$ and $(x_1,\dots,x_n)$ the dual basis associated with $(y_1,\dots,y_n)$. The \textit{rational Cherednik algebra} $H_{\h}(W)$ associated with $W=G(\l,1,n)$ is the quotient of $T(\cat{h}\oplus\cat{h}^*)\rtimes W$ by the following relations (see \cite[2.6]{go:qu:08}) :

$$[y_i,x_i]
=-h\sum_{r=0}^{\l-1}\sum_{j\neq i} s_{i,j}\sigma_i^r\sigma_j^{-r}- \sum_{t=1}^{\l-1}\left(\sum_{j=0}^{\l-1}\zeta_\l^{-tj}H_j\right) \sigma_i^t $$ 

$$ [y_i,x_j]=h\sum_{r=0}^{\l-1}\zeta_\l^r s_{i,j}\sigma_i^r\sigma_j^{-r} \text{ if } i\neq j.$$ By \cite[4.15]{et:sy:02}, we have an inclusion $\C[\cat{h}]^W\otimes\C[\cat{h}^*]^W\subset Z(H_{\h}(W))$, where $Z(H_{\h}(W))$ is the center of $H_{\h}(W)$
. This allows us to define the \textit{restricted rational Cherednik algebra} $$\overline{H_{\h}(W)}:=H_{\h}(W)/((\C[\cat{h}]^W\otimes\C[\cat{h}^*]^W)_+H_{\h}(W))$$ where $(\C[\cat{h}]^W\otimes\C[\cat{h}^*]^W)_+$ denotes the ideal in $\C[\cat{h}]^W\otimes\C[\cat{h}^*]^W$ of elements with zero constant term. The PBW property (see \cite[1.3]{et:sy:02}) implies that $$\overline{H_{\h}(W)}\simeq \C[\cat{h}]^{coW}\otimes_\C\C[W]\otimes_\C\C[\cat{h}^*]^{coW}$$ as vector spaces, where $\C[\cat{h}]^{coW}:=\C[\cat{h}]/(\C[\cat{h}]^W)_+$ is the coinvariant ring of $\cat{h}$ with respect to $W$. In particular $\dim_\C \overline{H_{\h}(W)}=|W|^3$. 

Finally, it has been shown in \cite[4.3]{go:ba:03} that $\irr \overline{H_{\h}(W)}$ can be naturally identified with $\irr W$: $$\irr \overline{H_{\h}(W)}=\{L_\h(E),~E\in\irr W\}.$$ Therefore the blocks of the restricted rational Cherednik algebra induce a partitioning of the set $\irr(W)$ called \textit{the Calogero-Moser partition} and denoted by $CM_\h(W)$ (see \cite{go:ca:09} for more details). Moreover, we can define an ordering on $\mathcal{P}(\l,n)$ according to the value of the $\fc$-function that assigns to $\llambda\in\mathcal{P}(\l,n)$ the scalar $\fc_{\h}(\llambda)$ by which the Euler element $eu\in Z(\overline{H_{\h}(W)})$ acts on the simple $\overline{H_{\h}(W)}$-module $L_{\h}(E_\lambda)$: $$\llambda\lq_c\mmu \Longleftrightarrow \fc_\h(\llambda)<\fc_\h(\mmu) \text{ or } \llambda=\mmu.$$
 
\end{abs}

\begin{abs}{\bf Cyclotomic Hecke algebras.} \hspace{0,05cm} 
Let $r\in\Z_{>0}$, $\m=(m^0,\dots,m^{\l-1})\in\Q^\l$ such that $rm^i\in\Z$ for all $i=0,\dots, \l-1$ and let $\q$ be an indeterminate. The \textit{cyclotomic Hecke algebra} $\mathcal{H}_{\q,\m}(W)$ corresponding to $W=G(\l,1,n)$ is the associative algebra over $\C[\q^{\pm 1}]$ which is free as a $\C[\q^{\pm 1}]$-module, with basis $(T_0,\dots,T_{n-1})$ satisfying the braid relations: $$\begin{array}{rcll} T_iT_j&=&T_jT_i & \text{ if }|i-j|>1,\\
 T_iT_{i+1}T_i&=&T_{i+1}T_iT_{i+1} & \text{ if }1\leqslant i\leqslant n-2,\\ 
T_0T_1T_0T_1&=&T_1T_0T_1T_0, \end{array}$$ and the following relations: 
$$\prod_{j=0}^{\l-1}(T_0-\zeta_{\l}^j\q^{rm^j})=0 \text{ and } (T_i-\q^{r})(T_i+1)=0 \text{ for } 1\leqslant i\leqslant n-1.$$

Let $L$ be the field generated by the traces on $\cat{h}$ of all the elements of $W$, let $\mu(L)$ be the group of all the roots of unity of $L$ and let $\z$ be an indeterminate such that $\q=\z^{|\mu(L)|}$. According to \cite[Prop 4.3.4]{ch:bl:09}, $\C(\z)\otimes_{\C[\q^{\pm 1}]}\mathcal{H}_{\q,\m}(W)$ is split semisimple. Let $\llambda\in\mathcal{P}(\l,n)$ and let $s_{\lambda}\in\C[\z^{\pm 1}]$ be the Schur element of $\C(\z)\otimes_{\C[\q^{\pm 1}]}\mathcal{H}_{\q,\m}(W)$ associated with $E_\lambda$ (see \cite[Chap. 7]{ge:ch:00}). We set
$\a_{\m,r}(\llambda):=-val_{\q}(s_\lambda)=-\frac{val_{\z}(s_\lambda)}{k}$. In the same manner as the $\fc$-order, we will consider the ordering on $\mathcal{P}(\l,n)$ induced by the $\a$-function: $$\llambda\lq_{\mbi{a}}\mmu \Longleftrightarrow \a_\h(\llambda)<\a_\h(\mmu) \text{ or } \llambda=\mmu.$$

\begin{Rem}
For $r<0$ we will use the formula of \cite[4.2]{ch:sh:09}: $\a_{\m,r}(\llambda)=\a_{-\m,-r}(^t\llambda)$ to define the $\a$-function, with $^t\llambda=(^t\lambda^1,\dots,^t\lambda^\l)$.
\end{Rem}

\end{abs}

\begin{abs}{\bf Combinatorial orderings.}\label{comb-ord} \hspace{0,05cm}  Let us consider two other orderings defined on $\mathcal{P}(\l,n)$.

 The first one is associated with the symbols of a multipartition (see  \cite[\S 5.5]{ge:re:10}). Let $\m=(m^0,\dots,m^{\l-1})$ $\in\Q^\l$, $\llambda=(\lambda^1,\dots,\lambda^\l)\in\mathcal{P}(\l,n)$ and, for $1\lq i\lq \l$, let $h^i:=h(\lambda^i)$ be the height of the partition $\lambda^i$. For $0\lq i\lq \l-1$, we put $hc^i:=h^{i+1}-m^i$ and $hc^\lambda:=\max(hc^0,\dots,hc^{\l-1})$ and we denote by $[t]$ the integer part of a non-negative rational number $t$. Let $s$ be an integer such that $s\gq hc^\lambda+1$, we call the \textit{shifted $\m$-symbol of $\llambda$ of size $s$}: $$\cat{B}_{\m}^s(\llambda)=\begin{pmatrix}\cat{B}^0\\
\vdots\\
\cat{B}^{\l-1}\end{pmatrix}$$ with $\cat{B}^i=\beta^i(s-hc^{i})$ for $0\leqslant i\leqslant \l-1$, 
 where $\beta^i=(\lambda^{i+1}_{h^{i+1}},\dots,\lambda_j^{i+1}-j+h^{i+1},\dots,\lambda_1^{i+1}-1+h^{i+1})$ and for $\beta=(\beta_1,\dots,\beta_k)$ and $t\geqslant 0$, 
$$\beta(t):=\left\{\begin{array}{ll} (\beta_1,\dots,\beta_k) & \text{ if }0\leqslant t< 1\\
(t-[t],t-[t]+1,\dots,t-1,\beta_1+t,\dots,\beta_k+t) &  \text{ if } t\geqslant1.\end{array} \right.$$

\begin{Exa}
Let $\l=2$, $n=5$, $\llambda=(\emptyset;(3,2))$ and $s=4$, $$\cat{B}^s_{(\frac{1}{2},0)}(\llambda)=\left(\begin{array}{cccc}\frac{1}{2} & \frac{3}{2} & \frac{5}{2} &\frac{7}{2} \\ 0 & 1 & 4 & 6 \end{array}\right).$$
\end{Exa} 

Let $\kappa_1\gq\cdots\gq\kappa_t$ be the elements of $\cat{B}_{\m}^s(\llambda)$ written in decreasing order (allowing repetition), with $t=\l s+\ds{\sum_{i=0}^{\l-1}m^i}$. We set \begin{equation}\label{kappaetN} \kappa_\m^s(\llambda)=(\kappa_1,\dots,\kappa_t)\text{ and }N_{\m}^s(\llambda)=\l\sum_{i=1}^t\sum_{j=0}^{[\kappa_i]}(\kappa_i-j)=\l\sum_{i=1}^t\frac{[\kappa_i]+1}{2}\left(2\kappa_i-[\kappa_i]\right).\end{equation}

\begin{Rem}
(i) It is easily seen that the rational number $\ds{\sum_{i=1}^t \kappa_i}$ does not depend on $\llambda$. We still use the term of "partition" for $\kappa_\m^s(\llambda)$, even if the $\kappa_i$'s are not integers. Moreover, we extend the definition of dominance order $\unlhd$ to sequences of rational numbers in a obvious way. This allows us to define an order relation on $\mathcal{P}(\l,n)$ by comparing the associated partitions $\kappa_\m^s(\llambda)$: $$\llambda\leqslant_\kappa \mmu \Longleftrightarrow \kappa_\m^s(\llambda)\lhd \kappa_\m^s(\mmu) \text{ or } \llambda=\mmu.$$

\nd (ii) The group $\cat{S}_\l$ acts naturally on $\mathcal{P}(\l,n)$ and $\Q^\l$ in the following way: let $w\in\cat{S}_\l$, $\llambda=(\lambda^1,\dots,\lambda^\l)\in\mathcal{P}(\l,n)$ and $(q_1,\dots,q_\l)\in\Q^\l$, $$w\cdot\llambda=(\lambda^{w^{-1}(1)},\dots,\lambda^{w^{-1}(\l)}) \text{ and } w\cdot(q_1,\dots,q_\l)=(q_{w^{-1}(1)},\dots,q_{w^{-1}(\l)})$$ and by definition, we have $\kappa_{w\cdot \m}^s(w\cdot \llambda)=\kappa_\m^s(\llambda)$.
\end{Rem}

The second ordering we define on $\mathcal{P}(\l,n)$ is associated with an other partition, its construction is due to \cite[\S 6]{go:qu:08}. 
Let $\rho$ be a partition and let $s\in\Z$. We associate an infinite set of strictly decreasing integers : $$\beta_s(\rho)=(\rho_1+s,\rho_2+s-1,\dots,\rho_j+s-j+1,\dots).$$ It is clear that we can recover $\rho$ and $s$ from this set since the sequence stabilizes to $s+1-j$ for $j\gq h(\rho)+1$. Let $\llambda\in\mathcal{P}(\l,n)$ and $\s\in\Z_0^\l:=\ds{\{(s_0,\dots,s_{\l-1})\in\Z^\l,~\sum_{i=0}^{\l-1}s_i=0\}}$.
For $1\leqslant i\leqslant \l$, we set $\mathcal{T}_i$ to be the decreasing infinite sequence :

$$\mathcal{T}_i:=\{(\l(x-1)+i,~x\in\beta_{s_{i-1}}(\lambda^i)\}$$ The partition $\tau_{\is}(\llambda)$ is defined as the unique partition $\rho$ such that $\beta_0(\rho)$ is the set $\displaystyle{\mathcal{T}=\bigcup_{i=1}^{\l}\mathcal{T}_i}$ written in decreasing order. This procedure yields a bijection
 $$\tau~:~\left\{\begin{array}{rcl}\ds{ \Z_0^{\l}\times\bigsqcup_{M\gq 0}\mathcal{P}(\l,M) }& \rightarrow &\ds{ \bigsqcup_{N\gq 0}\mathcal{P}(N)}\\

(\s,\llambda)& \mapsto & \tau_{\is}(\llambda). \end{array}\right.$$

\begin{Rem}
If we define, for $w\in \cat{S}_\l$ and $\s=(s_0,\dots,s_{\l-1})\in\Z_0^\l$, $w\cdot \s=(s_0',\dots,s_{\l-1}')$ where for $1\leqslant i\leqslant \l$ $s_{i-1}'=s_{w^{-1}(i)-1}+\frac{w^{-1}(i)-i}{\l}$ then we can still define $\tau_{w\cdot\is}(w\cdot\llambda)$ (because $\ds{\sum_{i=0}^{\l-1}s_i'=0}$ and $\l s_i'\in\Z^\l$ for all $i$) and, by definition, $\tau_{w\cdot\is}(w\cdot\llambda)=\tau_{\is}(\llambda)$. This action of $\cat{S}_\l$ on  $\Z_0^\l$ could look mysterious but we will see later (see \ref{comb-desc}) that in fact it is natural.
\end{Rem}

\begin{Exa}

\nd$\tau_{(1,-1)}((2,2,1);\emptyset)=(5,4,1,1)$ because $\beta_1((2,2,1))=(3,2,0,-2,-3,$ $-4,\dots)$ and $\beta_{-1}(\emptyset)= (-1,-2,-3,-4,\dots)$ thus

$$\begin{array}{rcl}\mathcal{T}_1&=&\{2(3-1)+1, 2(2-1)+1,2(0-1)+1,2(-2-1)+1,2(-3-1)+1,\dots\}\\
&=&\{5,3,-1,-5,-7,-9,\dots\} \end{array}$$

$$\begin{array}{rcl}\mathcal{T}_2&=&\{2(-1-1)+2, 2(-2-1)+2,2(-3-1)+2,\dots\} \\
&=&\{-2,-4,-6,-8,\dots\} \end{array}$$

and $\beta_0(\tau_{(1,-1)}((2,2,1);\emptyset))=(5,3,-1,-2,-4,-5,-6,-7,-8,\dots)$.\end{Exa}

Our next theorem compares the different orderings we defined on $\mathcal{P}(\l,n)$.

\begin{Th}

Let $\m\in\Q^\l$, $s$ an integer such that $s\geqslant \max\{hc^\lambda,hc^\mu\}+1$ and $r\in\Z_{>0}$.

\begin{enumerate}

\item \cite[5.5.16]{ge:re:10}   $\kappa^s_{\m}(\llambda)\triangleleft\kappa^s_{\m}(\mmu)~ \Rightarrow ~\a_{\m,r}(\llambda)> \a_{\m,r}(\mmu)$.

 \item If $\h=(h,H_1,\dots,H_{\l-1})$ is such that $h=r$ and for $1\lq i\lq \l-1$, $H_1+\cdots+H_i=rm^i$ then $$\kappa^s_{\m}(\llambda)\triangleleft\kappa^s_{\m}(\mmu) ~\Rightarrow ~\fc_{\h}(\llambda)< \fc_{\h}(\mmu).$$
 
 \item Let $\s=(s_0,\dots,s_{\l-1})\in\Z_0^\l$. If $\m$ is defined by $m^i=-s_i-\frac{i}{\l}$ for $0\leqslant i\leqslant \l-1$, then $$\tau_{\is}(^t\llambda)\vartriangleleft\tau_{\is}(^t\mmu)~\Longleftrightarrow~  \kappa^s_{\m}(\mmu)\vartriangleleft\kappa^s_{\m}(\llambda).$$
 
 \item Let $\s=(s_0,\dots,s_{\l-1})\in\Z_0^\l$. If $\m$ is defined by $m^i=s_{\l-1-i}-\frac{i}{\l}$ for $0\leqslant i\leqslant \l-1$, then $$\tau_{\is}(^t\llambdabar)\vartriangleleft\tau_{\is}(^t\mmubar)~\Longleftrightarrow~  \kappa^s_{\m}(\llambda)\vartriangleleft\kappa^s_{\m}(\mmu).$$

 \end{enumerate}

\end{Th}

\begin{Rem}
The relations $h=r$ and $H_1+\cdots+H_i=rm^i$ for $1\lq i\lq \l-1$ are the same as  $\m_{(\mathcal{C},j)}=\h_{(\mathcal{C},j)}$ given in \cite[3.3]{ch:ce:11} with the following changes of the parameter space  : $h=\h_{(\mathcal{C}_1,0)}-\h_{(\mathcal{C}_1,1)}$, $H_i=\h_{(\mathcal{C}_2,i)}-\h_{(\mathcal{C}_2,i-1)}$ for all $1\lq i\lq \l-1$ (see  \cite[3.1]{ma:ca:10}), $r=\m_{(\mathcal{C}_1,0)}-\m_{(\mathcal{C}_1,1)} \text{ and for } 0\lq j\lq \l-1,~m^j=\frac{\m_{(\mathcal{C}_2,j)}}{r}.$ From now on, we will denote equally $\a_{\m,r}$, $\a_\h$, $\fc_{\m,r}$ or $\fc_\h$, considering this change of parameters.
\end{Rem}

\newpage

\begin{proof}\hspace{0,5cm}

\begin{enumerate}

 \setcounter{enumi}{1}

\item Thanks to the formula of  \cite[6.2]{ro:q:08}, it is easy to show that the $\fc$-order is the same as the ordering associated to the function $N_{\m}^s(\llambda)$ we defined in (\ref{kappaetN}) (see \cite{li:al:12} for more details). Thus it is sufficient to show that if $\kappa^s_{\m}(\llambda)\triangleleft~\kappa^s_{\m}(\mmu)$ then $N_{\m}^s(\llambda)<N_{\m}^s(\mmu)$. Let us suppose that there does not exist a partition $\tilde\kappa$ such that $\kappa^s_{\m}(\llambda)\triangleleft~\tilde\kappa~\lhd\kappa^s_{\m}(\mmu)$. It follows from \cite[1.4.10]{ja:th:81} that there exists $j>i$ such that : $\kappa^s_{\m}(\llambda)=(\kappa_1,\dots,\kappa_{i-1},\kappa_i,\kappa_{i+1},\dots,\kappa_{j-1},\kappa_j,\kappa_{j+1},\dots,\kappa_t)$ and $\kappa^s_{\m}(\mmu)=(\kappa_1,\dots,\kappa_{i-1},\kappa_i',\kappa_{i+1},\dots,\kappa_{j-1},\kappa_j',\kappa_{j+1},\dots,\kappa_t)$, with $\kappa_i'=\kappa_i+\alpha$ and $\kappa_j'=\kappa_j-\alpha$, where $\alpha>0$. We can write

 $$\frac{N_{\m}^s(\mmu)-N_{\m}^s(\llambda)}{\l}=f_\alpha(\kappa_i)-f_\alpha(\kappa_j'),$$ with $f_\alpha~:~x\in\R^+\mapsto \frac{[x+\alpha]-[x]}{2}(2x+2\alpha-[x+\alpha]-[x]-1)+\alpha x$ a strictly increasing function. The conclusion follows easily.

\item Let $w_0:=(1, \l)(2, \l-1)\dots$ be the longest element of $\cat{S}_\l$. Using the second remark of \ref{comb-desc} and the fact that $\tau_{\bar\is}({\llambdabar})=^t(\tau_{\is}(\llambda))$ where $\bar\s=(-s_{\l-1},\dots,-s_0)$ and ${\llambdabar}=(^t\lambda^\l,\dots,^t\lambda^1)$, we obtain $$\tau_{\is}(^t\llambda)=\tau_{\is}(w_0\cdot\llambdabar)=\tau_{w_0\cdot\is}(\llambdabar)=^t\tau_{\overline{w_0\cdot\is}}(\llambda),$$ with $(\overline{w_0\cdot\s})_i=-s_i+\frac{\l-1-2i}{\l}.$ Then we have to compare the sequences 
 $\beta_0(\tau_{\overline{w_0\cdot\is}}(\llambda))$ and $\kappa^s_{\m}(\llambda)$ for chosen parameters.

\nd But $\beta_0(\tau_{\overline{w_0\cdot\is}}(\lambda))$ corresponds to $\displaystyle{\bigcup_{i=1}^{\l}\mathcal{T}_i}$ where, for $1\leqslant i\leqslant \l$, $$\mathcal{T}_i= \left\{\l\left(\lambda_j^i-j-s_{i-1}-\frac{i}{\l}\right)+\l-1,~ j\geqslant 1  \right\} .$$

\nd Let $m^i=-s_i-\frac{i}{\l}$ and $s\geqslant\max\{h(\lambda^i)-m^{i-1},h(\mu^i)-m^{i-1},~1\leqslant i\leqslant \l\}+1$, we obtain :

$$\begin{array}{rcl}\frac{\mathcal{T}_i-(\l-2)}{\l}+s & = & (\lambda^i_1-1+m^{i-1}+s,\dots,\lambda^i_{h(\lambda^i)}-h(\lambda^i)+m^{i-1}+s,\\
 & & m^{i-1}+s-h(\lambda^i)-1,\dots,m^{i-1}+s-h(\lambda^i)-[m^{i-1}+s-h(\lambda^i)],\\
 & & m^{i-1}+s-h(\lambda^i)-[m^{i-1}+s-h(\lambda^i)]-1,\dots)\end{array}$$
 
 \nd consequently $$\frac{\mathcal{T}_i-(\l-2)}{\l}+s=\tilde{\cat{B}}^{i-1},$$ where $\tilde{\cat{B}}^i=(\cat{B}^i_{s+[m^i]},\dots,\cat{B}^i_1,\cat{B}^i_1-1,\cat{B}^i_1-2,\dots)$ if $\cat{B}^i=(\cat{B}^i_1,\cdots,\cat{B}^i_{s+[m^i]})$.
 
 \nd Therefore, if we denote by $\widetilde{\kappa_{\m}^s}(\llambda)$ the ordered sequence corresponding to $\widetilde{\cat{B}}=(\widetilde{\cat B}^0,\dots,\widetilde{\cat B}^{\l-1})$, we have $$\beta_0(\tau_{\overline{w_0\cdot\is}}(\llambda))+\l s=\l\widetilde{\kappa_{\m}^s}(\llambda)+\l-2$$

\nd and hence :

$$\begin{array}{rcl}\tau_{\is}(^t\llambda)\vartriangleleft\tau_{\is}(^t\mmu)&\Longleftrightarrow&    \tau_{\overline{w_0\cdot\is}}(\mmu)\lhd\tau_{\overline{w_0\cdot\is}}(\llambda)   \\
&\Longleftrightarrow&    \beta_0(\tau_{\overline{w_0\cdot\is}}(\mmu))\lhd\beta_0(\tau_{\overline{w_0\cdot\is}}(\llambda) )  \\
&\Longleftrightarrow& \widetilde{\kappa_{\m}^s}(\mmu)\vartriangleleft\widetilde{\kappa^s_{\m}}(\llambda) \\
&\Longleftrightarrow& \kappa_{\m}^s(\mmu)\vartriangleleft\kappa_{\m}^s(\llambda).\end{array}$$

\item Since $\tau_{\is}(^t\llambdabar)=\tau_{\is}(w_0\cdot\lambda)$ and $s_{i-1}+\frac{i}{\l}=m^{\l-i}+1=(w_0\cdot \m)^{i-1}+1$ for all $0\lq i\lq \l-1$, in this case the set $\mathcal{T}_i$ is such that $$\frac{\mathcal{T}_i}{\l}-1+s=\tilde{\cat{B}}^{w_0(i)-1}$$ and we can conclude in the same manner of 3.

\end{enumerate}

\end{proof}

\end{abs}

\begin{abs}{\bf Quiver varieties.}\label{qu-var} \hspace{0,05cm}  Let $Q$ be the cyclic quiver with $\l$ vertices $0,\dots,\l-1$, let $Q_{\infty}$ be the quiver obtained by adding one vertex named $\infty$ to $Q$ that is joined to $0$ by a single arrow from $\infty$ to $0$. Let $\overline{Q_\infty}$ be the quiver obtained by inserting an arrow in the opposite direction to every arrow in $Q_\infty$.

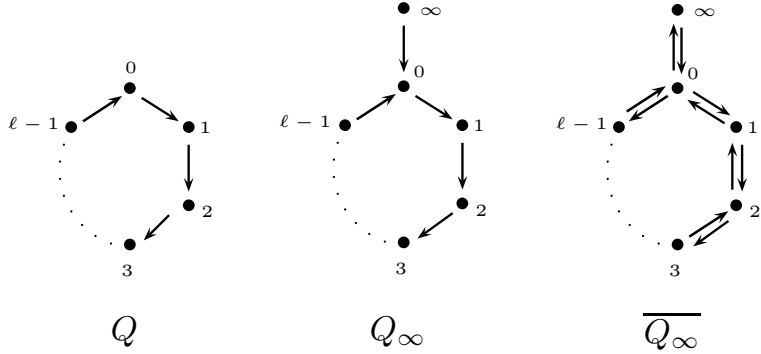
\begin{center}
\scalebox{1.3} 
{
\begin{pspicture}(0,-1.8401562)(6.7171874,1.8401562)
\psdots[dotsize=0.12](0.619375,0.91984373)
\psdots[dotsize=0.12](1.219375,0.51984376)
\psdots[dotsize=0.12](1.219375,-0.28015625)
\psdots[dotsize=0.12](0.619375,-0.68015623)
\psdots[dotsize=0.12](0.019375,0.51984376)
\psline[linewidth=0.025999999cm,arrowsize=0.05291667cm 2.0,arrowlength=1.4,arrowinset=0.4]{->}(0.139375,0.57984376)(0.539375,0.83984375)
\psline[linewidth=0.024cm,arrowsize=0.05291667cm 2.0,arrowlength=1.4,arrowinset=0.4]{->}(0.739375,0.81984377)(1.119375,0.57984376)
\psline[linewidth=0.024cm,arrowsize=0.05291667cm 2.0,arrowlength=1.4,arrowinset=0.4]{->}(1.219375,0.33984375)(1.219375,-0.14015625)
\psline[linewidth=0.024cm,arrowsize=0.05291667cm 2.0,arrowlength=1.4,arrowinset=0.4]{->}(1.019375,-0.38015625)(0.759375,-0.64015627)
\psarc[linewidth=0.024,linestyle=dotted,dotsep=0.16cm](0.699375,0.07984375){0.8}{147.38075}{263.3675}
\usefont{T1}{ptm}{m}{n}
\rput(0.5607813,-1.5701562){$Q$}
\usefont{T1}{pcr}{m}{n}
\rput(-0.3690625,0.5498437){\tiny $\l-1$}
\usefont{T1}{pcr}{m}{n}
\rput(0.63,1.1298437){\tiny $0$}
\usefont{T1}{pcr}{m}{n}
\rput(1.3871875,0.50984377){\tiny $1$}
\usefont{T1}{pcr}{m}{n}
\rput(1.4101562,-0.35015625){\tiny $2$}
\usefont{T1}{pcr}{m}{n}
\rput(0.5921875,-0.9501563){\tiny $3$}

\psdots[dotsize=0.12](3.419375,0.9398438)
\psdots[dotsize=0.12](4.019375,0.53984374)
\psdots[dotsize=0.12](4.019375,-0.26015624)
\psdots[dotsize=0.12](3.419375,-0.66015625)
\psdots[dotsize=0.12](2.819375,0.53984374)
\psline[linewidth=0.025999999cm,arrowsize=0.05291667cm 2.0,arrowlength=1.4,arrowinset=0.4]{->}(2.939375,0.59984374)(3.339375,0.85984373)
\psline[linewidth=0.024cm,arrowsize=0.05291667cm 2.0,arrowlength=1.4,arrowinset=0.4]{->}(3.539375,0.83984375)(3.919375,0.59984374)
\psline[linewidth=0.024cm,arrowsize=0.05291667cm 2.0,arrowlength=1.4,arrowinset=0.4]{->}(4.019375,0.35984376)(4.019375,-0.12015625)
\psline[linewidth=0.024cm,arrowsize=0.05291667cm 2.0,arrowlength=1.4,arrowinset=0.4]{->}(3.919375,-0.36015624)(3.559375,-0.6201562)
\psarc[linewidth=0.024,linestyle=dotted,dotsep=0.16cm](3.499375,0.09984375){0.8}{147.38075}{263.3675}
\usefont{T1}{ptm}{m}{n}
\rput(3.3607812,-1.5701562){$Q_\infty$}
\usefont{T1}{pcr}{m}{n}
\rput(2.4089061,0.56984377){\tiny $\l-1$}
\usefont{T1}{pcr}{m}{n}
\rput(3.5748436,1.0898438){\tiny $0$}
\usefont{T1}{pcr}{m}{n}
\rput(4.1871877,0.52984375){\tiny $1$}
\usefont{T1}{pcr}{m}{n}
\rput(4.2101564,-0.33015624){\tiny $2$}
\usefont{T1}{pcr}{m}{n}
\rput(3.3921876,-0.93015623){\tiny $3$}
\psdots[dotsize=0.12](3.419375,1.7398437)
\psline[linewidth=0.024cm,arrowsize=0.05291667cm 2.0,arrowlength=1.4,arrowinset=0.4]{->}(3.419375,1.5598438)(3.419375,1.0798438)
\usefont{T1}{pcr}{m}{n}
\rput(3.7,1.7098438){\tiny $\infty$}
\usefont{T1}{pcr}{m}{n}

\psdots[dotsize=0.12](6.219375,0.91984373)
\psdots[dotsize=0.12](6.819375,0.51984376)
\psdots[dotsize=0.12](6.819375,-0.28015625)
\psdots[dotsize=0.12](6.219375,-0.68015623)
\psdots[dotsize=0.12](5.619375,0.51984376)
\psline[linewidth=0.025999999cm,arrowsize=0.05291667cm 2.0,arrowlength=1.4,arrowinset=0.4]{->}(5.699375,0.6798437)(6.099375,0.9398438)
\psline[linewidth=0.024cm,arrowsize=0.05291667cm 2.0,arrowlength=1.4,arrowinset=0.4]{->}(6.379375,0.8798438)(6.759375,0.63984376)
\psline[linewidth=0.024cm,arrowsize=0.05291667cm 2.0,arrowlength=1.4,arrowinset=0.4]{->}(6.879375,0.33984375)(6.879375,-0.14015625)
\psline[linewidth=0.024cm,arrowsize=0.05291667cm 2.0,arrowlength=1.4,arrowinset=0.4]{->}(6.739375,-0.44015625)(6.379375,-0.7001563)
\psarc[linewidth=0.024,linestyle=dotted,dotsep=0.16cm](6.299375,0.07984375){0.8}{147.38075}{263.3675}
\usefont{T1}{ptm}{m}{n}
\rput(6.1607813,-1.5701562){$\overline{Q_\infty}$}
\usefont{T1}{pcr}{m}{n}
\rput(5.238906,0.5498437){\tiny $\l-1$}
\usefont{T1}{pcr}{m}{n}
\rput(6.374844,1.0698438){\tiny $0$}
\usefont{T1}{pcr}{m}{n}
\rput(6.9871873,0.50984377){\tiny $1$}
\usefont{T1}{pcr}{m}{n}
\rput(7.010156,-0.35015625){\tiny $2$}
\usefont{T1}{pcr}{m}{n}
\rput(6.1921877,-0.9501563){\tiny $3$}
\psdots[dotsize=0.12](6.219375,1.7198437)
\psline[linewidth=0.024cm,arrowsize=0.05291667cm 2.0,arrowlength=1.4,arrowinset=0.4]{->}(6.259375,1.5398438)(6.259375,1.0598438)
\usefont{T1}{pcr}{m}{n}
\rput(6.5,1.6898438){\tiny $\infty$}
\usefont{T1}{pcr}{m}{n}
\psline[linewidth=0.024cm,arrowsize=0.05291667cm 2.0,arrowlength=1.4,arrowinset=0.4]{<-}(6.179375,1.5798438)(6.179375,1.0998437)
\psline[linewidth=0.025999999cm,arrowsize=0.05291667cm 2.0,arrowlength=1.4,arrowinset=0.4]{<-}(5.719375,0.57984376)(6.119375,0.83984375)
\psline[linewidth=0.024cm,arrowsize=0.05291667cm 2.0,arrowlength=1.4,arrowinset=0.4]{<-}(6.319375,0.7998437)(6.699375,0.5598438)
\psline[linewidth=0.024cm,arrowsize=0.05291667cm 2.0,arrowlength=1.4,arrowinset=0.4]{<-}(6.779375,0.35984376)(6.779375,-0.12015625)
\psline[linewidth=0.025999999cm,arrowsize=0.05291667cm 2.0,arrowlength=1.4,arrowinset=0.4]{->}(6.339375,-0.60015625)(6.739375,-0.34015626)
\end{pspicture} 
}

\end{center}

\nd We consider the quiver varieties related to the quiver $\overline{Q}_\infty$, following \cite[\S 3]{go:qu:08}.

\nd Let $\mathbf{R}(n):=\mat_{n}(\C)^\l\oplus\mat_{n}(\C)^\l \oplus\C^{n}\oplus(\C^{n})^*.$ The group $\G(n):=\gl_n(\C)^\l$ acts on this space in the following way : for $g=(g_0,\dots,g_{\l-1})\in\G(n)$ and $(\X,\Y;v,w)=(X_0,\dots,X_{\l-1},Y_0,\dots,Y_{\l-1}; $ $v,w)\in \mathbf R(n)$ : $$g\cdot(\X,\Y;v,w)=(g_1X_0g_0^{-1},\dots,g_0X_{\l-1}g_{\l-1}^{-1},g_0Y_0g_1^{-1},\dots,g_{\l-1}Y_{\l-1}g_0^{-1};g_0v,wg_0^{-1}).$$ Associated to some symplectic form on $\mathbf R(n)$, let $\mu_\C$ be the $\G(n)$-equivariant moment map defined by $$\mu_\mathbb{C}~:~\left\{\begin{array}{rcl}  \mathbf R(n) & \rightarrow & \lie(\G(n))\\
 (\X,\Y;v,w) & \mapsto & [\X,\Y]+vw\end{array} \right.$$ (here we identify $\lie(\G(n))$ with its dual using the trace pairing). Given $\ttheta=(\theta_0,\dots,\theta_{\l-1})\in\Q^\l$, we can introduce a complex variety :
the geometric invariant theory quotient $$\mathcal{M}_\theta(n)=\mu_\C^{-1}(0)//_\theta \G(n).$$ Let us explain this notation. 
  For $\ttheta\in\Z^\l,$ let $\chi_\theta$ be the character of $\G(n)$ defined by $\chi_\theta(g):=\ds{\prod_{i=0}^{\l-1}(\det g_i)^{\theta_i}}$ and for $\ttheta\in\Q^\l$ : $$\C[\mu_\C^{-1}(0)]^{\chi_{j\theta}}:=\{f\in\C[\mu_\C^{-1}(0)],~\forall g\in \G(n),~g\cdot f=\chi_{j\theta}(g)f\}$$ if $j\ttheta\in\Z^\l$ and $\C[\mu_\C^{-1}(0)]^{\chi_{j\theta}}:=0$ if not.
 The variety $\mathcal{M}_\theta(n)$ is defined by 
    $$\mathcal{M}_\theta(n):=\proj \bigoplus_{j\geqslant 0}\C[\mu_\C^{-1}(0)]^{\chi_{j\theta}}$$ and is projective over $\mathcal{M}_0(n)\simeq(\cat{h}\times\cat{h}^*)/W$ (see \cite[\S3.9]{go:qu:08}). Therefore, when $\mathcal{M}_\theta(n)$ is smooth, this gives a symplectic resolution of this singular quotient variety. An other link between these varieties and the algebraic problems we consider is the fact that, if $\l=1$, then the Cherednik algebra $H_{\h}(W)$ provides a quantization of the Hilbert scheme of $n$ points on the plane, the relevant quiver variety in this special case.\\
    
Let us now describe some regularity properties of these varieties. The following result was proved by \cite[4.4]{go:qu:08}.

\begin{Lem}[Gordon]\label{murs-git} Let $\h=(h,H_1,\dots,H_{\l-1})\in\Q^\l$ and set $\ttheta=(-h+H_0,H_1,\dots,H_{\l-1})$, where $H_0=-(H_1+\cdots+H_{\l-1})$.
Then the variety $\M_{\theta}(n)$ is smooth if $\ttheta$does not lie on one of the following hyperplanes $$h=0 \text{ or } (H_i+\cdots+H_j)+mh=0$$ where $1\lq i\lq j\lq \l-1$ et $1-n\lq m\lq n-1$.
\end{Lem}

\nd We denote by $\H^\text{reg}$ the open subset of $\H:=\{\h=(h,H_1,\dots,H_{\l-1})\in\Q^\l\}$ obtained by removing the hyperplanes occurring in this lemma. Abusing terminology, we will call them the \textit{GIT walls} and consider that they define the \textit{GIT chambers} inside of which the corresponding varieties are smooth and isomorphic. 

\vspace{0,6cm}

\begin{center}

\begin{pspicture}(0,0)(12,2.3) 

\psellipticarc[linestyle=dashed](9,0.1)(1.5,0.25){0}{180} 
\psellipticarc(9,0.1)(1.5,0.25){180}{360} 
\psellipse(9,2.3)(1.5,0.25) 
\psline(7.5,2.25)(10.5,0.1) 
\psline(10.5,2.25)(7.5,0.1) 

\psellipticarc[linestyle=dashed](3,0.1)(1.5,0.25){0}{180}
\psellipticarc(3,0.1)(1.5,0.25){180}{360}
\psellipse(3,2.3)(1.5,0.25)
\pscurve(1.5,2.25)(2.2,1.175)(1.5,0.1)
\pscurve(4.5,2.25)(3.8,1.175)(4.5,0.1)
\psellipticarc[linestyle=dashed](3,1.175)(0.8,0.125){0}{180}
\psellipticarc(3,1.175)(0.8,0.125){180}{360}

\psline{->}(5,1.175)(7.5,1.175)


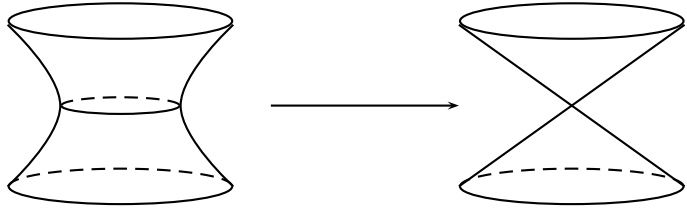
\rput(6,-0.6){$\M_{\theta}(1)$ and $\M_0(1)\simeq \C^2/\mu_2$, for $\h\in\H^\text{reg}$ and $W=G(2,1,1)$.} 
\end{pspicture}

\vspace{0,8cm}

\end{center}

\nd Let $\ttheta$ and $\ttheta'$ be two parameters such that : \vspace{-0,5cm}
 \begin{center}
            \colorbox{gris25}{
                    \fbox{\begin{minipage}{0,977\textwidth}  \begin{itemize}
\item $\ttheta'$ is on a GIT wall
\item $\ttheta$ belongs to a chamber in front of this wall 
\end{itemize} \end{minipage}
                    }
            }
    \end{center}
then, according to the following result, the variety $\M_{\theta}(n)$ is a resolution of singularities of $\M_{\theta'}(n)$.

\begin{Th}
There exists a projective and surjective morphism $\pi_{\theta,\theta'}~:~\mathcal{M}_\theta(n)\rightarrow \mathcal{M}_{\theta'}(n).$
\end{Th}

\begin{proof}
Some results of Nakajima give others descriptions of quiver varieties. Indeed, \cite[3.3]{re:mo:08} and \cite[2.9 (1)]{na:qu:09} claim that $\M_\theta(n)=\mu_\C^{-1}(0)^{\theta\text{-}s}/\G(n)$ and $\M_{\theta'}(n)={\mu_\C^{-1}(0)}^{\theta'\text{-}ss}//_{\theta'}\G(n)$ (for the definitions of stability and semistability, see \cite[2.4 (2)]{na:qu:09}). Moreover \cite[2.12 (1) and (3)(a)]{na:qu:09} show that ${\mu_\C^{-1}(0)}^{\theta\text{-}s}={\mu_\C^{-1}(0)}^{\theta\text{-}ss}\subset{\mu_\C^{-1}(0)}^{\theta'\text{-}ss}$, hence the morphism $\pi_{\theta,\theta'}~:~\mathcal{M}_\theta(n)\rightarrow \mathcal{M}_{\theta'}(n)$ exists. This morphism is such that the diagram $$\xymatrix{\mathcal{M}_\theta(n) \ar[rr]^{\pi_{\theta,\theta'}} \ar[rd]_{\pi_{\theta,0}} & & \mathcal{M}_{\theta'}(n) \ar[ld]^{\pi_{\theta',0}} \\ & \mathcal{M}_0(n)\simeq (\cat{h}\times\cat{h}^*)/W & }$$ commutes. Therefore, since $\pi_{\theta,0}$ and $\pi_{\theta',0}$ are projective, $\pi_{\theta,\theta'}$ is projective. For the surjectivity,  \cite[2.24]{na:qu:09} claims that all strata in $\pi_{\theta,\theta'}(\M_\theta(n))$ are relevant in $\M_{\theta'}(n)$ for a quiver of affine type.
\end{proof}

\nd For parameters $\h'\in\H-\H^\text{reg}$ and $\ttheta'=(-h+H_0,H_1,\dots,H_{\l-1})$, we do not know if $\M_{\theta'}(n)$ is regular but the next result shows that this variety is normal.

\begin{prop}\label{normale}
For all $\h\in\H$ and $\ttheta=(-h+H_0,H_1,\dots,H_{\l-1})$, $\mathcal{M}_{\theta}(n)$ is a normal variety.
\end{prop}

\begin{proof}
According to \cite[1.1]{cr:no:03}, we know that $\M_0(n)$ is normal. But, \cite[(2.3)]{na:qu:09} shows that :
$\mathcal{M}_0(n)=M^\text{el}\times\mathcal{M}_0^\text{norm}$, where $M^\text{el}$ is an affine space. Thus $\mathcal{O}_{\mathcal{M}_0}=\mathcal{O}_{\mathcal{M}_0^\text{norm}}[X_1,\dots,X_n]$. Combining this with the fact that $\mathcal{O}_{\mathcal{M}_0}$ is integrally closed, we obtain that ${\mathcal{M}_0^\text{norm}}$ is a normal variety.
Moreover, \cite[\S 2.7]{na:qu:09} shows that $\mathcal{M}_{\theta}(n)$ is locally isomorphic to $\hat{\mathcal{M}}_0^\text{norm}\times T$ where $\hat{\mathcal{M}}_0^\text{norm}$ is normal (because this variety is very close to  ${\mathcal{M}_0^\text{norm}}$: only the spaces $(V,W)$ change) and  $T$ is the product of a vector space with an affine space. Therefore $\M_{\theta}(n)$ is a normal variety. \end{proof}

\end{abs}


\section{Geometric ordering on $CM_{\h}(G(\l,1,n))$ for $\h\in\H$.}

Throughout this section we will assume that $\h$ and $\ttheta$ are related by the usual rule  $\ttheta=(-h+H_0,H_1,\dots,H_{\l-1})$. Using the two previous sections, we are now able to give a geometric description of $CM_\h(G(\l,1,n))$, thanks to the quiver varieties $\M_\theta(n)$, and we can build a geometric ordering on $CM_\h(G(\l,1,n))$ for all $\h\in\Q^\l$. Then we will relate it with combinatorial and algebraic orders already defined on $\mathcal{P}(\l,n)$.

\begin{abs}{\bf Geometric description of $CM_{\h}(G(\l,1,n))$ for $\h\in\H$.}\label{geo-des} \hspace{0,05cm} There is a $\C^*$-action on $\M_\theta(n)$ induced by the following hyperbolic action on $\mathbf R(n)$ : $\eta\cdot(\X,\Y;v,w)=(\eta\X,\eta^{-1}\Y;v,w)$.

\begin{Th}[Gordon]
 Let $\h=(h,H_1,\dots,H_{\l-1})\in\H$. Then the $\C^*$-fixed points on $\M_\theta(n)$ are naturally labeled by the Calogero-Moser blocks of $G(\l,1,n)$. Moreover, if $\H\in\H^\text{reg}$ then the blocks are singletons.
\end{Th}
\nd For the proof, based on the existence of a $\C^*$-equivariant homeomorphism between $\M_{\theta}(n)$ and the Calogero-Moser space $\spec Z(H_{\h}(W))$, we refer the reader to \cite[5.1]{go:qu:08}.

\vspace{0,6cm}

\hspace{4,5cm}\begin{pspicture}(0,0)(12,2.3) 

\psellipticarc[linestyle=dashed](3,0.1)(1.5,0.25){0}{180}
\psellipticarc(3,0.1)(1.5,0.25){180}{360}

\psellipse(3,2.3)(1.5,0.25)

\pscurve(1.5,2.25)(2.2,1.175)(1.5,0.1)
\pscurve(4.5,2.25)(3.8,1.175)(4.5,0.1)

\psellipticarc[linestyle=dashed](3,1.175)(0.8,0.125){0}{180}
\psellipticarc(3,1.175)(0.8,0.125){180}{360}

\rput(2.2,1.175){$\textcolor{gray}\bullet$} 
\rput(3.8,1.175){\large$\bullet$} 


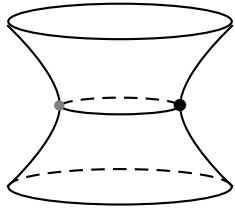
\rput(3,-0.6){The two fixed points of $\M_{\theta}(1)$, for $W=G(2,1,1)$ and $\h\in\H^\text{reg}$.} 
\end{pspicture}


        
    \vspace{1cm}

 \end{abs}

 \begin{abs}{\bf Definition.} \hspace{0,05cm}  According to Theorem \ref{para1} and Proposition \ref{qu-var}, for all $\h\in\H$, the transitive relation generated by the rule $$\llambda\prec_{\h} \mmu~\Longleftrightarrow ~
\overline{\mathcal{Z}_\theta(n)_{x_\theta(\lambda)}}\cap \mathcal{Z}_\theta(n)_{x_\theta(\mu)}\neq \emptyset, $$ is a partial order, where $\mathcal{Z}_\theta(n):=\pi_{\theta,0}^{-1}\Big((\cat{h}\times\{0\})/W\Big)$ and $\mathcal{Z}_\theta(n)_{x_\theta(\lambda)}:=\{z\in\mathcal{Z}_\theta(n),~\lim_{\eta\rightarrow 0} \eta\cdot z=x_\theta(\llambda)\}$ is the attractive set associated to $x_\theta(\llambda)$.

\begin{Rem}
The construction in the smooth case was made by Gordon in \cite[5.4]{go:qu:08}. We generalize his work for all $\h$ thanks to the results of the section \ref{para1}.
\end{Rem}

\vspace{0,6cm}
\hspace{4,5cm}\begin{pspicture}(0,0)(12,2.3) 

\psellipticarc[linestyle=dashed](3,0.1)(1.5,0.25){0}{180}
\psellipticarc(3,0.1)(1.5,0.25){180}{360}
\psellipse(3,2.3)(1.5,0.25)
\pscurve(1.5,2.25)(2.2,1.175)(1.5,0.1)
\pscurve(4.5,2.25)(3.8,1.175)(4.5,0.1)

\rput(2.2,1.175){$\textcolor{gray}\bullet$} 

\psellipticarc[linestyle=dotted, linecolor=gray](3,1.175)(0.8,0.125){0}{180}
\psellipticarc[linecolor=gray](3,1.175)(0.8,0.125){180}{360}

\rput(3.8,1.175){\large$\bullet$} 

\psline[linewidth=2pt](3.8,1.175)(4.05,2.1)
\psline[linestyle=dashed, linewidth=2pt](3.8,1.175)( 3.55,0.3)
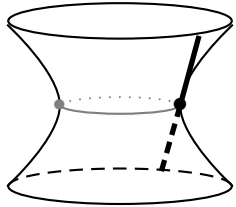
\rput(3,-0.6){Attractive sets of $\M_{\theta}(1)$, for $W=G(2,1,1)$ and $\h\in\H^\text{reg}$.} 
\end{pspicture}

        
    \vspace{1cm}
    
    \nd We see in this picture that, for $\l=2$ and $n=1$, we have \textcolor{gray}{gray point} $\prec_\h$ \textbf{black point}.

 \end{abs}

\begin{abs}{\bf Relations with other orders for $\h\in\H^\text{reg}$.}\label{comb-desc} \hspace{0,05cm}  Abusing notation, from now on, we will consider $\H=\{(h,H_1,\dots,H_{\l-1})\in \Q^\l,~h\neq 0\}$. According to \cite[7.12]{go:qu:08}, we have a partition of $\H^\text{reg}$ in alcoves, where the alcove $\alpha(\s,w,+)$ corresponds to the chamber that contains $\ttheta_0^+=w^{-1}(\mathbf{1}+(s_0-s_{\l-1},s_1-s_0,\dots,s_{\l-1}-s_{\l-2}))$ and $\alpha(\s,w,-)$ to the chamber that contains $\ttheta_0^-=w^{-1}(-\mathbf{1}+(s_{\l-1}-s_{0},s_{\l-2}-s_{\l-1},\dots,s_{0}-s_{1}))$, where $\ds{\mathbf{1}=\left(\frac{1}{\l},\dots,\frac{1}{\l}\right)}\in\Z^\l$, $\s\in\Z_0^\l$ and $w\in\cat{S}_\l$.

\begin{Rem}
The action of $\cat{S}_\l$ on $\Z_0^\l$ we defined in \ref{comb-ord} is such that, for $w'\in\cat{S}_\l$, $\alpha(w'\cdot \s,w'w,\pm)=\alpha(\s,w,\pm)$. Moreover, if we set $\tthetabar=(-\theta_0,-\theta_{\l-1},\dots,-\theta_1)$, for $\ttheta=(\theta_0,\dots,\theta_{\l-1})$, then we have :
$\ttheta\in\alpha(\s,id,+)~\Longleftrightarrow~{\tthetabar}\in\alpha(\s,id,-)$. This will allows us to consider only the alcoves of the form $\alpha(\s,id,+)$ from now on.
\end{Rem}

\begin{prop}[Gordon]
Let $\ttheta\in\alpha(\s,id,+)$ and let $\nu_{\is}$ be the $\l$-core associating with $\s\in\Z_0^\l$ (see \cite[6.4]{go:qu:08}). We set $N=\l n+|\nu_{\is}|$. Then we have two bijections $$\begin{array}{rcccl} \mathcal{P}(\l,n) & \leftrightarrow & (\M_\theta(n))^{\C^*} & \leftrightarrow & \mathcal{P}_{\nu_{\is}}(N) \\
\llambda & \mapsto & x_\theta(\llambda) & \mapsto & \tau_{\is}(^t\llambda) \end{array}$$ where $ \mathcal{P}_{\nu_{\is}}(N)$ is the set of the partitions of $N$ whose $\l$-core is $\nu_{\is}$ and $\tau_{\is}(^t\llambda)$ is the partition we defined in  \ref{comb-ord}.
\end{prop}

\nd The first bijection comes from Theorem \ref{geo-des} and it is compatible with the $\bar{~~}$-involution in the following sense. By \cite[7.1]{go:qu:08}, there exists an isomorphism $\phi~:~\M_{\theta}(n)\rightarrow \M_{\bar\theta}(n)$ such that $\phi(x_\theta(\llambda))=x_{\bar\theta}(\llambdabar)$. The second bijection was proved in \cite[Lem. 7.8 and prop. 7.10]{go:qu:08}.

\nd Let us now recall the definition of the combinatorial ordering made in \cite[\S 7]{go:qu:08}. It depends on the alcove containing the parameter $\ttheta$ in the following way :

\begin{itemize} \item for $\ttheta\in\alpha(s,w,+)$, $$\llambda\lhd_\theta \mmu~\Longleftrightarrow~ \tau_{\is}(w\cdot ^t \mmu)\lhd \tau_{\is}(w\cdot^t\llambda),$$
\item for $\ttheta\in\alpha(s,w,-)$, $$\llambda\lhd_\theta \mmu~\Longleftrightarrow~ \tau_{\is}(w\cdot ^t\mmubar)\lhd \tau_{\is}(w\cdot^t\llambdabar).$$
\end{itemize}

\begin{Rem}
By definition and according to the remark \ref{comb-desc}, we have $$\llambda\lhd_\theta\mmu~\Longleftrightarrow~\llambdabar\lhd_{\overline\theta}\mmubar.$$
\end{Rem}

This combinatorial description will allow us to show that the geometric ordering is interesting because it is in relation with many other orders which appear in the representation theory of reflection groups. The left part of the next result : $$\begin{array}{cl}\llambda\prec_{\theta}\mmu & \Longrightarrow \llambda\lhd_{\theta}\mmu\\
 \Downarrow & \\
 \a_{\h}(\llambda)<\a_{\h}(\mmu) & \\
 \fc_{\h}(\llambda)>\fc_{\h}(\mmu) \end{array}$$ comes from \cite[5.4, 7.12, 9.3]{go:qu:08}.

\begin{Th}\label{theo2} Let $\h\in\H^\text{reg}$ and $\ttheta=(-h+H_0,H_1,\dots,H_{\l-1})$. Then, for $\s\in\Z_0^\l$ and $w\in\cat{S}_\l$ :

\begin{itemize}

\item if $\ttheta\in\alpha(\s,w,-)$, then

\hspace{-0,7cm} $\begin{array}{cl}\llambda\prec_{\theta}\mmu & \Longrightarrow \llambda\lhd_{\theta}\mmu \xLeftrightarrow{\stackrel{(w\cdot m)^i=}{s_{\l-1-i}-\frac{i}{\l}}} \kappa_{\m}^s(\mmu)\lhd\kappa_{\m}^s(\llambda) \xRightarrow{r>0} \left\{\begin{array}{rcl} \a_{\m,r}(\llambda)&<&\a_{\m,r}(\mmu)\\
 \fc_{\m,r}(\llambda)&>&\fc_{\m,r}(\mmu) \end{array}\right.\\
 \Downarrow & \\
 \a_{\h}(\llambda)<\a_{\h}(\mmu) & \\
 \fc_{\h}(\llambda)>\fc_{\h}(\mmu) \end{array}$ 
 
 \vspace{0,2cm}

\item if $\ttheta\in\alpha(\s,w,+)$, then

\hspace{-0,5cm} $\begin{array}{cl}\llambda\prec_{\theta}\mmu & \Longrightarrow \llambda\lhd_{\theta}\mmu \xLeftrightarrow{\stackrel{(w\cdot m)^i=}{-s_i-\frac{i}{\l}}} \kappa_{\m}^s(\llambda)\lhd\kappa_{\m}^s(\mmu) \xRightarrow{r<0} \left\{\begin{array}{rcl} \a_{\m,r}(\llambda)&<&\a_{\m,r}(\mmu)\\
 \fc_{\m,r}(\llambda)&>&\fc_{\m,r}(\mmu) \end{array}\right.\\
 \Downarrow & \\
 \a_{\h}(\llambda)<\a_{\h}(\mmu) & \\
 \fc_{\h}(\llambda)>\fc_{\h}(\mmu) \end{array}$

 \end{itemize}
 
 \nd where $s$ is an integer such that $s\gq\max\{hc^\lambda,hc^\mu\}+1$.

\end{Th}

\begin{Rem}

(i) The relation obtained between the geometric and combinatorial orderings and the dominance order on the partitions $\kappa_{\m}^s(\llambda)$ could be interesting because this order is related, at least in some cases, with the order $\prec_{\mathcal{LR}}$, see \cite[2.6, 7.11]{ge:or:12}.

\nd (ii) The result we obtain for the $\a$ and $\fc$-functions is less important than the results of \cite[5.4, 9.3]{go:qu:08} which is for all parameters $\h$ but our proof is more direct and the combinatorial order is involved. 

\nd (iii)  For $\l=2$, $n=5$, $\ttheta\in A_0=\alpha((0,0),id,+)$, $\llambda=(\emptyset;(3,2))$, $\mmu=((2,2,1);\emptyset)$ and $s=4$, we have $\tau_{(0,0)}(^t\llambda)=(4,3,1,1,1) \lhd (5,2,2,1)=\tau_{(0,0)}(^t\mmu),$ hence $\mmu\lhd_\theta \llambda$. 

\nd Moreover $\kappa^s_{(\frac{1}{2},0)}(\mmu)=(5,5 ~;~ 4,5~; ~3 ~; ~2,5 ~; ~2 ~;~ 1 ~;~ 0,5 ~;~ 0)\lhd \kappa^s_{(\frac{1}{2},0)}(\llambda)=(6 ~;~ 4 ~;~ 3,5 ~;~ 2,5 ~;~ 1,5 ~;~ 1 ~;~ 0,5 ~;~ 0) $ but $\kappa^s_{\left(\frac{9}{10},0\right)}(\llambda)=(6 ~;~ 4 ~;~ 3,9~;~ 2,9 ~;~ 1,9 ~;~ 1 ~;~ 0,9 ~;~ 0)$ and $\kappa^s_{\left(\frac{9}{10},0\right)}(\mmu)=(5,9 ~;~ 4,9~; ~3 ~; ~2,9 ~; ~2 ~;~ 1 ~;~ 0,9 ~;~ 0)$ are incomparable and
$\kappa^s_{(1,0)}(\llambda)=(6 ~;~ 4 ~;~ 4 ~;~ 3 ~;~ 2 ~;~ 1 ~;~ 1 ~;~ 0 ~;~  0) \lhd (6 ~;~ 5 ~;~ 3~ ;~ 3 ~;~ 2~ ;~ 1~ ;~ 1 ~; ~0 ~;~ 0)=\kappa^s_{(1,0)}(\mmu).$ Thus the dominance order on the partitions $\kappa^s_\m(\llambda)$ is not constant on an alcove, that is why we can not compare it with $\lhd_\theta$ for all parameters.

\nd Finally, we have $\a_{\left(\frac{1}{2},0\right),1}(\llambda)=32,5<34=\a_{\left(\frac{1}{2},0\right),1}(\mmu)$ and $\a_{(1,0),1}(\llambda)=40>39=\a_{(1,0),1}(\mmu)$ thus by continuity of the $\a$-function with respect to the parameters $(\m,r)$(see \cite[Prop. 5.5.11]{ge:re:10}), we can not compare the multipartitions $\llambda$ and $\mmu$ with the geometric order in this alcove. That proves that $\prec_\theta$ and $\lhd_\theta$ are not equivalent and all the above implications are not equivalences.

\nd (iv) Using the formulas $h=r$ et $H_i=r(m^i-m^{i-1})$, we get 
\begin{itemize}
\item for $m^i=-s_i-\frac{i}{\l}$ and $r=-1$ : $\ttheta=\mathbf{1}+(s_0-s_{\l-1},s_1-s_0,\dots,s_{\l-1}-s_{\l-2}),$
\item for $m^i=s_{\l-1-i}-\frac{i}{\l}$ and $r=1$ : $\ttheta=-\mathbf{1}+(s_{\l-1}-s_{0},s_{\l-2}-s_{\l-1},\dots,s_{0}-s_{1})$
\end{itemize} and we recognize the parameters involved in the definition of alcoves, hence the change of parameters is consistent.

\end{Rem}

\begin{proof} Let us show that $$\llambda\lhd_{\theta}\mmu \xLeftrightarrow{\stackrel{(w\cdot m)^i=}{s_{\l-1-i}-\frac{i}{\l}}} \kappa_{\m}^s(\mmu)\lhd\kappa_{\m}^s(\llambda) \xRightarrow{r>0} \left\{\begin{array}{rcl} \a_{\m,r}(\llambda)&<&\a_{\m,r}(\mmu)\\
 \fc_{\m,r}(\llambda)&>&\fc_{\m,r}(\mmu). \end{array}\right.$$
 According to the first remark of \ref{comb-desc}, we just have to consider the case $\ttheta\in\alpha(\s,id,-)$. Then $\ds{-h=\sum_{i=0}^{\l-1}\theta_i<0}$ and $r=h>0$. Consequently,  according to Theorem \ref{comb-ord}, we have : $$\llambda\lhd_\theta\mmu~\Longleftrightarrow~\kappa^s_{\m}(\mmu)\vartriangleleft\kappa^s_{\m}(\llambda)~\Longrightarrow~\left\{\begin{array}{rcl} \a_{\m,r}(\llambda)&<&\a_{\m,r}(\mmu)\\
 \fc_{\m,r}(\llambda)&>&\fc_{\m,r}(\mmu). \end{array}\right.$$
 
 \nd For positive alcoves, we have to consider $\tthetabar$ that corresponds to $(\overline\m,-r)$, where $\overline\m=(-m^{\l-1},\dots,-m^0)$ with $-r>0$. If $m^i=-s_i-\frac{i}{\l}$ then $\overline\m^i=s_{\l-1-i}-\frac{i}{\l}+\frac{\l-1}{\l}$ and, since the dominance order on the partitions $\kappa_{\m}^s(\llambda)$ is invariant by translation on $\m$, we can apply the result we obtained for negative alcoves.

\end{proof}

\end{abs}

\begin{abs}\label{order}{\bf Combinatorial preorder and relations between the orders on $CM_\h(G(\l,1,n)$ for $\h\in\H-\H^\text{reg}$.} \hspace{0,05cm} In this section, we generalize the results of the previous section from $\H^\text{reg}$ to $\H$. It will be useful for the study of the groups $G(\l,e,n)$ for $e\neq 1$, see \S \ref{glen}.

Let us begin with the generalization of \cite[5.4, 9.3]{go:qu:08} for $\h'$ in a wall. Both the $\a$-function and the $\fc$-function are constant across blocks of Calogero-Moser (see \cite[5.3, 9.2]{go:qu:08}) therefore we can define $\a_{\h'}(B)$ and $\fc_{\h'}(B)$ for $B\in CM_{\h'}(G(\l,1,n))$.

\begin{Th}
Let $\h'\in\H$ and let $B,~B'\in CM_{\h'}(G(\l,1,n))$. Then $$B\prec_{\theta'} B' ~\Longrightarrow~\left\{\begin{array}{rcl}\a_{\h'}(B)&\lq&\a_{\h'} (B')\\
\fc_{\h'}(B)&\gq&\fc_{\h'} (B'). \end{array}\right.$$
\end{Th}

\begin{proof}
Let $\ttheta$ inside an alcove in front of the wall where $\ttheta'$ lives. According to Proposition \ref{theo1} and Theorem \ref{theo2} :
$$\begin{array}{rcl} B\prec_{\theta'} B' &\Longleftrightarrow & \exists~ \llambda_1,\llambda_2,\llambda_2',\llambda_3',\dots,\llambda_m\in\mathcal{P}(\l,n) ~:~\llambda_1\in B,~\llambda_m\in B',\\ 
& & \forall~2\lq i\lq m-1,~\llambda_i \text{ and } \llambda'_i \text{ are in the same block} \\
& & \forall~1\lq i\lq m-1,~\llambda_i\prec_{\theta} \llambda_{i+1}\text{ if } i\not\in 2\Z\text{ and } \llambda_i'\prec_{\theta} \llambda_{i+1}'\ \text{ if }i\in 2\Z\\

&\Longrightarrow & \exists~ \llambda_1,\llambda_2,\llambda_2',\llambda_3',\dots,\llambda_m\in\mathcal{P}(\l,n) ~:~\llambda_1\in B,~\llambda_m\in B',\\ 
& & \forall~2\lq i\lq m-1,~\a_{\h'}(\llambda_i )=\a_{\h'}( \llambda'_i) \\
& & \forall~1\lq i\lq m-1,~\a_{\h}(\llambda_i)< \a_\h(\llambda_{i+1})\text{ if } i\not\in 2\Z\text{ and }\a_\h(\llambda_i')<\a_\h( \llambda_{i+1}')\ \text{ if }i\in 2\Z.
\end{array}$$

\nd Since the $\a$-function is continuous with respect to $\h$, the fact that $\a_{\h}(\llambda_i)< \a_\h(\llambda_{i+1})$ and $\a_\h(\llambda_i')<\a_\h( \llambda_{i+1}')$ for every $\h$ in the alcove implies that $\a_{\h'}(\llambda_i)\lq \a_{\h'}(\llambda_{i+1})$ and $\a_{\h}(\llambda_i')\lq \a_{\h'}( \llambda_{i+1}')$. Hence :

$$\begin{array}{rcl} B\prec_{\theta'} B' &\Longrightarrow&  \exists~ \llambda_1,\llambda_2,\llambda_2',\llambda_3',\dots,\llambda_m\in\mathcal{P}(\l,n) ~:~\llambda_1\in B,~\llambda_m\in B',\\ 
& & \forall~2\lq i\lq m-1,~\a_{\h'}(\llambda_i )=\a_{\h'}( \llambda'_i) \\
& & \forall~1\lq i\lq m-1,~\a_{\h'}(\llambda_i)\lq \a_{\h'}(\llambda_{i+1})\text{ if } i\not\in 2\Z\text{ and }\a_{\h'}(\llambda_i')\lq\a_{\h'}( \llambda_{i+1}')\ \text{ if }i\in 2\Z\\

&\Longrightarrow&  \exists~ \llambda_1,\llambda_2,\llambda_2',\llambda_3',\dots,\llambda_m\in\mathcal{P}(\l,n) ~:~\llambda_1\in B,~\llambda_m\in B',\\ 

& &\text{and }\a_{\h'}(\llambda_1)\lq \a_{\h'}(\llambda_2)=\a_{\h'}(\llambda_2')\lq \a_{\h'}(\llambda_3')=\cdots \lq \a_{\h'}(\llambda_m)  \\

&\Longrightarrow&  \exists~ \llambda_1,\llambda_2,\llambda_2',\llambda_3',\dots,\llambda_m\in\mathcal{P}(\l,n) ~:~\llambda_1\in B,~\llambda_m\in B',\\ 

& & \text{and } \a_{\h'}(\llambda_1)\lq \a_{\h'}(\llambda_m)\\

&\Longrightarrow& \a_{\h'}(B)\lq \a_{\h'}(B').
\end{array}$$

\nd For the $\fc$-function, the proof is the same.
\end{proof}

Let us now give a combinatorial description of the geometric order. According to \cite[8.1]{go:qu:08}, an element $\h\in\H$ has a type $J\subset \{0,\dots,\l-1\}$ that corresponds to the fundamental hyperplanes in which lives the $\tilde{\cat{S}_\l}$-conjugate of $\h$ in the closure of the fundamental alcove $A_0$ (containing $\ttheta_0=\mathbf 1$). For instance, the elements inside an alcove has type $\emptyset$.

\nd Let $J\subset\{0,\dots,\l-1\}$ and $\rho\in\mathcal{P}(N)$, for some $N\in\Z_{>0}$. The \textit{$J$-heart} $\rho_J$ of $\rho$ is the subpartition of $\rho$ obtained by removing all $j$-removable boxes with $j\in J$ from $\rho$ (see \cite[8.2]{go:qu:08} for more details). The subset of $\mathcal{P}(N)$ whose elements are the partitions which have the same $J$-heart $\rho_J$ is called the \textit{$J$-class} of $\rho$ and is denoted by $\overline{\rho}^J$. We are now able to reformulate Proposition 8.3 of \cite{go:qu:08}.

\begin{prop}[Gordon]

Let $J\subset\{0,\dots,\l-1\}$ and suppose that $\ttheta'$ is of type $J$ and belongs to the closure of the alcove $\alpha(\s,w,+)$. Let $\ttheta\in\alpha(\s,w,+)$ and let $\nu_\s$ be the $\l$-core associating with $\s$. We set $N=\l n+|\nu_\s|$. Then the $\C^*$-fixed points of $\M_{\theta'}(n)$ are labelled by the $J$-classes in $\mathcal{P}_{\nu_s}(N)$ and the bijection is such that the restriction of $\pi_{\theta,\theta'}$ to the fixed points can be described as the application : $$\begin{array}{rcl} \mathcal{P}(\l,n) & \rightarrow & \mathcal{P}(\mathcal{P}_{\nu_s}(N))\\ \llambda & \mapsto & \overline{\tau_\s(^t\llambda)}^J \end{array}$$ where $ \mathcal{P}(\mathcal{P}_{\nu_s}(N))$ is the power set of  $\mathcal{P}_{\nu_s}(N)$.

\end{prop}

\nd We are now able to define a combinatorial preorder on $CM_{\h'}(G(\l,1,n))$. For $\ttheta'$ in a wall next to the alcoves $A_1,\dots,A_q$ and $\llambda,~\mmu\in\mathcal{P}(\l,n)$, we set : $$\begin{array}{c}\llambda\lhd_{\theta'}\mmu~\Longleftrightarrow~\exists~ \llambda_0=\llambda,\dots,\llambda_p=\mmu \text{ such that for all }0\lq i\lq p-1 :\\
\exists 1\lq j\lq q \text{ such that, for } \ttheta_j\in A_j,~\llambda_i\lhd_{\theta_j}\llambda_{i+1}.\end{array}$$
 We define \textit{the combinatorial preorder on $CM_{\h'}(G(\l,1,n))$} in the following way :  let $B,~B'\in CM_{\h'}(G(\l,1,n))$, $$B\lhd_{\theta'} B'~\Longleftrightarrow~\exists~\llambda\in B \text{ and } \mmu\in B', \text{ such that } \llambda\lhd_{\theta'}\mmu.$$
 
  Then we can relate it with the geometric ordering for $\l=2$ i.e. for the Weyl groups of type $B_n$. The next result compares the Calogero-Moser blocks of $B_n$ with the equivalence classes of this preorder.

\begin{prop} Assume that $\ttheta'$ is a wall and let $\ttheta$ and $\tthetatilde$ be two parameters which belong to each side of this wall.
If $\llambda$ and $\mmu$ are in the same block for $\ttheta'$ then there exists $ \llambda_0=\llambda,\llambda_1,\dots,\llambda_p=\mmu $ such that for all $ 1\leqslant i\leqslant p -1:$  $$(\llambda_i\lhd_{\theta}\llambda_{i+1} \text{ or } \llambda_i\lhd_{\widetilde \theta}\llambda_{i+1}) \text{ and }  (\llambda_{i+1}\lhd_{\theta}\llambda_i \text{ or }  \llambda_{i+1}\lhd_{\widetilde \theta}\llambda_i).$$
\end{prop}

\begin{proof} The alcoves for $\l=2$ are described in \cite[3.5]{go:ca:09}. We set $\ttheta=(1-\frac{b}{a},\frac{b}{a})$ and $\m=(\frac{b}{a},0)$, then for $i\in\Z$ let $A_i=\{(d,-d+1),~i<d<i+1\}$. If $i\in 2\Z$ then $A_i=\alpha((\frac{i}{2},-\frac{i}{2}),id,+)$ otherwise $A_i=\alpha((\frac{1-i}{2},\frac{i-1}{2}),\sigma,+)$, where $\sigma$ is the transposition $(12)$.

\begin{Lem}\label{blocs-symboles} Let $\m=\left(\frac{b}{a},0\right)$ with $\frac{b}{a}\in\Z$ and $\m\pm\frac{1}{2}=\left(\frac{b}{a}\pm\frac{1}{2},0\right)$. If $\llambda$ and $\mmu\in\mathcal{P}(2,n)$ are such that $\llambda\neq\mmu$ and if $s\gq n+1$ then :
$$\begin{array}{rcl}\kappa^s_{\m}(\llambda)=\kappa^s_{\m}(\mmu) & \Longrightarrow &\exists~ \llambda_0=\llambda,\llambda_1,\dots,\llambda_p=\mmu \text{ such that } \forall ~0\lq i\lq p-1 : \\ & & \kappa^s_{\m+\frac{1}{2}}(\llambda_i)\lhd\kappa^s_{\m+\frac{1}{2}}(\llambda_{i+1}) \text{ and } \kappa^s_{\m-\frac{1}{2}}(\llambda_{i+1})\lhd \kappa^s_{\m-\frac{1}{2}}(\llambda_i) \text{ or }\\
& & \kappa^s_{\m+\frac{1}{2}}(\llambda_i)\rhd\kappa^s_{\m+\frac{1}{2}}(\llambda_{i+1}) \text{ and } \kappa^s_{\m-\frac{1}{2}}(\llambda_{i+1})\rhd \kappa^s_{\m-\frac{1}{2}}(\llambda_i). \end{array}$$
\end{Lem}

\begin{proof}[Proof of the lemma :]

If $\kappa^s_{\m}(\llambda)=\kappa^s_{\m}(\mmu)$ then $\cat{B}^s_{\m}(\llambda)$ and $\cat{B}^s_{\m}(\mmu)$ have the same elements. We first suppose that we can obtain $\cat{B}^s_{\m}(\mmu)$ from $\cat{B}^s_{\m}(\llambda)$ by permuting two elements and show that \begin{itemize}
\item either $\kappa^s_{\m+\frac{1}{2}}(\llambda)\lhd\kappa^s_{\m+\frac{1}{2}}(\mmu)$ and $\kappa^s_{\m-\frac{1}{2}}(\mmu)\lhd\kappa^s_{\m-\frac{1}{2}}(\llambda)$ 
\item or $\kappa^s_{\m+\frac{1}{2}}(\llambda)\rhd\kappa^s_{\m+\frac{1}{2}}(\mmu)$ and $\kappa^s_{\m-\frac{1}{2}}(\mmu)\rhd\kappa^s_{\m-\frac{1}{2}}(\llambda)$. \end{itemize} In the general case, the sequence of multipartitions $\llambda_0,\dots,\llambda_p$ will be built step by step.

\nd Let  $$\cat{B}^s_{\m}(\llambda)=\left(\begin{array}{cccc}x_1,&\dots&\dots,&x_{s+\frac{b}{a}}\\ y_1, & \dots, &y_s & \end{array}\right)$$ 

\nd where, for $1\leqslant k\leqslant s-h(\lambda^1)+\frac{b}{a}$, $x_k=k-1$ and for $1\leqslant k\leqslant h(\lambda^1)$, $x_{s+\frac{b}{a}+1-k}=\lambda_k^1-k+s+\frac{b}{a}$ are such that
$x_1<x_2<\cdots<x_{s+\frac{b}{a}}$, 

\nd and for $1\leqslant k\leqslant s-h(\lambda^2)$, $y_k=k-1$ and for $1\leqslant k\leqslant h(\lambda^2)$, $y_{s+1-k}=\lambda_k^2-k+s$ are such that $y_1<y_2<\cdots<y_s$.     

\nd Let us suppose that $\cat{B}^s_{\m}(\mmu)$ differs from $\cat{B}^s_{\m}(\llambda)$ by a permutation $x_i\leftrightarrow y_j$, with $x_i<y_j$. Therefore for all $1\leqslant k\leqslant s$, $y_k\neq x_i$ and :


$$\kappa^s_{\m}(\llambda)=(\kappa_1\geqslant\cdots\geqslant\kappa_{s_1}>\kappa_{s_1+1} >\kappa_{s_1+2}\geqslant\cdots\geqslant \kappa_{s_2} > \kappa_{s_2+1}>\kappa_{s_2+2}\geqslant\cdots
\geqslant\kappa_{2s+\frac{b}{a}})$$

 \nd where $\kappa_{s_1+1}=y_j$, $\kappa_{s_2+1}=x_i$ and $\kappa_{2s+\frac{b}{a}-1}=\kappa_{2s+\frac{b}{a}}=0$.\\

 \nd Since $[s-h(\lambda^1)+\frac{b}{a}-\frac{1}{2}]=s-h(\lambda^1)+\frac{b}{a}-1$, we have  $\cat{B}^s_{\m-\frac{1}{2}}(\llambda)=(\cat{B}^0,\cat{B}^1)$, where $\cat{B}^0=(x_2-\frac{1}{2},\dots,x_{s+\frac{b}{a}}-\frac{1}{2})$ has $s+\frac{b}{a}-1$ terms and $\cat{B}^1=(y_1,\dots,y_s)$. Moreover the elements in the first row of $\cat{B}^s_{\m+\frac{1}{2}}(\llambda)$ are obtained by adding $\frac{1}{2}$ to those of $\cat{B}^s_{\m}(\llambda)$.

\nd Since the $\kappa_i$ are integers, we have :  
 $$\kappa^s_{\m+\frac{1}{2}}(\llambda)=(\kappa'_1\geqslant\cdots\geqslant\kappa'_{s_1}>\kappa_{s_1+1} >\kappa'_{s_1+2}\geqslant\cdots\geqslant \kappa'_{s_2} > \kappa_{s_2+1}+\frac{1}{2}>\kappa'_{s_2+2}\geqslant\cdots
\geqslant\kappa'_{2s+\frac{b}{a}}),$$

\nd where $\kappa'_i=\left\{\begin{array}{rl}\kappa_i+\frac{1}{2} & \text{ if it is on the first row of }\cat{B}^s_{\m}(\llambda)\\
\kappa_i & \text{ otherwise} \end{array}\right.$ and

$$\kappa^s_{\m+\frac{1}{2}}(\mmu)=(\kappa'_1\geqslant\cdots\geqslant\kappa'_{s_1}>\kappa_{s_1+1}+\frac{1}{2} >\kappa'_{s_1+2}\geqslant\cdots\geqslant \kappa'_{s_2} > \kappa_{s_2+1}>\kappa'_{s_2+2}\geqslant\cdots
\geqslant\kappa'_t).$$

\nd And then it is easily seen that $\kappa^s_{\m+\frac{1}{2}}(\llambda)\lhd\kappa^s_{\m+\frac{1}{2}}(\mmu)$ and $\kappa^s_{\m-\frac{1}{2}}(\mmu)\lhd\kappa^s_{\m-\frac{1}{2}}(\llambda)$.\\

\nd If $x_i>y_j$ then by the same manner we obtain : 
$\kappa^s_{\m+\frac{1}{2}}(\mmu)\lhd\kappa^s_{\m+\frac{1}{2}}(\llambda)$ and $\kappa^s_{\m-\frac{1}{2}}(\llambda)\lhd\kappa^s_{\m-\frac{1}{2}}(\mmu).$
\end{proof}

\nd To give the proof of the proposition, we will study the case of a wall $\frac{b}{a}=r+1$, with $r$ even (the proof of the odd case is the same). The parameter $\ttheta'=(-r,r+1)$ belongs to a wall in the right of $A_{-r}=\alpha((\frac{-r}{2},\frac{r}{2}),id,+)$ and in the left of $A_{-r-1}=\alpha((\frac{r}{2}+1,-\frac{r}{2}-1),\sigma,+)$. 

\nd According to the previous lemma and the results of \cite[3.13]{ma:ca:10} and \cite[3.4]{br:fa:02}, we have, if $\m=\left(\frac{b}{a},0\right)$ then :  $$\begin{array}{rcl}\llambda \text{ and } \mmu \text{ are in the same block } & \Longleftrightarrow & \kappa^s_{\m}(\llambda)=\kappa^s_{\m}(\mmu) \\ 
& \Longrightarrow & \exists ~\llambda_0=\llambda,\llambda_1,\dots,\llambda_p=\mmu \text{ such that } \forall i : \\ & & [\kappa^s_{\m+\frac{1}{2}}(\llambda_i)\lhd\kappa^s_{\m+\frac{1}{2}}(\llambda_{i+1})\\
& & \text{and } \kappa^s_{\m-\frac{1}{2}}(\llambda_{i+1})\lhd \kappa^s_{\m-\frac{1}{2}}(\llambda_i)] \text{ or }\\
& & [\kappa^s_{\m+\frac{1}{2}}(\llambda_i)\rhd\kappa^s_{\m+\frac{1}{2}}(\llambda_{i+1})\\
& & \text{and } \kappa^s_{\m-\frac{1}{2}}(\llambda_{i+1})\rhd \kappa^s_{\m-\frac{1}{2}}(\llambda_i)]. \end{array}$$

\nd The parameter $\m-\frac{1}{2}$ is associated with $\ttheta\in A_{-r}=\alpha((\frac{-r}{2},\frac{r}{2}),id,+)$ thus we have to put in relation $\kappa^s_{\m-\frac{1}{2}}(\llambda_i)$ and $\tau_{(\frac{-r}{2},\frac{r}{2})}(^t\llambda_i)$. But $\m-\frac{1}{2}=(r+1-\frac{1}{2},0)=(\frac{r}{2}+(\frac{r}{2}+\frac{1}{2}),-\frac{r}{2}-\frac{1}{2}+(\frac{r}{2}+\frac{1}{2}))$, hence according to Theorem \ref{comb-ord} we have :

$$\begin{array}{rcl} \kappa^s_{\m-\frac{1}{2}}(\llambda_{i+1})\lhd \kappa^s_{\m-\frac{1}{2}}(\llambda_i) & \Longleftrightarrow & \kappa^s_{(\frac{r}{2},-\frac{r}{2}-\frac{1}{2})}(\llambda_{i+1})\lhd \kappa^s_{(\frac{r}{2},-\frac{r}{2}-\frac{1}{2})}(\llambda_i) \\
& \Longleftrightarrow & \tau_{(\frac{-r}{2},\frac{r}{2})}(^t\llambda_i)\lhd  \tau_{(\frac{-r}{2},\frac{r}{2})}(^t\llambda_{i+1}) \\
& \Longleftrightarrow  & \llambda_{i+1}\lhd_\theta \llambda_i, \end{array}$$

\nd because $\s=(\frac{-r}{2},\frac{r}{2})$ and $\m=(\frac{r}{2},-\frac{r}{2}-\frac{1}{2})$ satisfy $m^i=-s_i-\frac{i}{2}$.

\nd The same argument for $\m+\frac{1}{2}$ and $\tthetatilde\in A_{-r-1}=\alpha((\frac{r}{2}+1,-\frac{r}{2}-1),\sigma,+)$ shows that  :

$$\begin{array}{rcl} \kappa^s_{\m+\frac{1}{2}}(\llambda_{i})\lhd \kappa^s_{\m+\frac{1}{2}}(\llambda_{i+1}) & \Longleftrightarrow & \kappa^s_{(0,r+\frac{3}{2})}(\sigma\cdot\llambda_{i})\lhd \kappa^s_{(0,r+\frac{3}{2})}(\sigma\cdot\llambda_{i+1}) \\
& \Longleftrightarrow  &  \kappa^s_{(-\frac{r}{2}-1,\frac{r}{2}+\frac{1}{2})}(\sigma\cdot\llambda_{i})\lhd \kappa^s_{(-\frac{r}{2}-1,\frac{r}{2}+\frac{1}{2})}(\sigma\cdot\llambda_{i+1}) \\
& \Longleftrightarrow  & \tau_{(\frac{r}{2}+1,-\frac{r}{2}-1)}(^t\sigma\cdot\llambda_{i+1}) \lhd  \tau_{(\frac{r}{2}+1,-\frac{r}{2}-1)}(^t\sigma\cdot\llambda_{i}) \\
& \Longleftrightarrow  & \llambda_{i}\lhd_{\widetilde \theta} \llambda_{i+1}. \end{array}$$

\nd Then we have : $$\begin{array}{rcl}\llambda \text{ and } \mmu \text{ are in the same block } & \Longrightarrow & \exists~ \llambda_0=\llambda,\llambda_1,\dots,\llambda_p=\mmu \text{ such that } \forall i : \\ & & (\llambda_i\lhd_{\widetilde \theta}\llambda_{i+1} \text{ and }\llambda_{i+1}\lhd_\theta \llambda_i )\\
& & \text{or } (\llambda_{i+1}\lhd_{\widetilde \theta}\llambda_i \text{ et }\llambda_i \lhd_\theta \llambda_{i+1}), \\
& \Longleftrightarrow & \exists ~\llambda_0=\llambda,\llambda_1,\dots,\llambda_p=\mmu \text{ such that } \forall i : \\ & & (\llambda_i\lhd_\theta\llambda_{i+1} \text{ or } \llambda_i\lhd_{\widetilde \theta}\llambda_{i+1}) \text{ and }  \\ & & (\llambda_{i+1}\lhd_\theta\llambda_i \text{ or }  \llambda_{i+1}\lhd_{\widetilde \theta}\llambda_i).  \end{array}$$

\end{proof}

\begin{Rem}
(i) A direct corollary of this proposition is : \begin{center} if $\llambda \text{ and } \mmu \text{ are in the same block for }\ttheta' \text{ in a wall}$ then $\llambda\lhd_{\theta'}\mmu \text{ and }\mmu\lhd_{\theta'}\llambda.$\end{center}

\nd (ii) The converse is not true. Indeed, for the same data as in the remark following Theorem \ref{comb-desc}, we have $\mmu\lhd_{\theta} \llambda$ and $\llambda\lhd_{\tilde\theta}\mmu$ for $\tthetatilde\in\alpha((1,-1),\sigma,+)$, since $\tau_{(1,-1)} (^t\sigma\cdot\llambda)=(5,4,1,1) \rhd (4, 2, 2, 2, 1)=\tau_{(1,-1)}(^t\sigma\cdot\mmu)$. But $\kappa_{(1,0)}^s(\llambda)\neq\kappa_{(1,0)}^s(\mmu)$ thus $\llambda$ and $\mmu$ are not in the same block for $\ttheta'=(0,1)$.

\nd (iii) This result improves the definition of the combinatorial preorder for $\l=2$ : $$B\lhd_{\theta'} B'~\Longleftrightarrow~\forall~\llambda\in B \text{ et } \forall ~\mmu\in B',~\llambda\lhd_{\theta'}\mmu.$$

\nd (iv) We hope that this proposition is still true for $\l>2$ but the proof should be harder. Indeed, we do not have an easy description of alcoves for $\l>2$ and the proof of our proposition is based on the fact that, for $\l=2$, the parameter $\m\pm\frac{1}{2}$ is a translation of $(-s_i-\frac{i}{\l})_i$ and that is not true in the general case. Nevertheless, a generalization for $\l>2$ of this result would be very interesting in order to generalize the following result.
\end{Rem}

\begin{Cor}
Let $\h'\in \H$ and let $B$, $B'$ be two blocks of $CM_{\h'}(G(2,1,n))$ then we have : $$B\prec_{\theta'} B' ~\Longrightarrow~B\lhd_{\theta'} B'.$$
\end{Cor}

\begin{proof} Let $\ttheta$ be inside an alcove next to the wall where $\ttheta'$ lives. According to Proposition \ref{theo1} and Theorem \ref{theo2} :
$$\begin{array}{rcl} B\prec_{\theta'} B'
& \Longleftrightarrow & \exists~ \llambda_1\in B, \llambda_2, \llambda_2',\llambda_3',\dots\text{ and }\llambda_m\in B' \text{ such that : }\\
& & \llambda_i \text{ and }\llambda'_i \text{ are in the same block for all }i>1\\
& & \llambda_i\prec_{\theta} \llambda_{i+1} \text{ for } i \text{ odd} \\
& & \llambda'_i\prec_{\theta} \llambda'_{i+1} \text{ for } i \text{ even}. \\
\end{array}$$

\nd But according to Theorem \ref{comb-desc}, $\llambda_i\prec_{\theta} \llambda_{i+1} \Longrightarrow \llambda_i\lhd_{\theta} \llambda_{i+1} $, hence using the previous remark, we obtain : 

$$\begin{array}{rcl} B\prec_{\theta'} {B'} & \Longrightarrow & \exists~ \llambda_1\in B, \llambda_2, \llambda_2',\llambda_3',\dots\text{ and }\llambda_m\in B' \text{ such that : }\\
& & \llambda_i \lhd_{\theta'}\llambda'_i \text{ and } \llambda_i'\lhd_{\theta'}\llambda_i \text{ for all }i>1\\
& & \llambda_i\lhd_{\theta} \llambda_{i+1} \text{ for } i \text{ odd} \\
& & \llambda'_i\lhd_{\theta} \llambda'_{i+1} \text{ for } i \text{ even}, \\\\

& \Longrightarrow & \exists~ \llambda_1\in B, \llambda_2, \llambda_2',\llambda_3',\dots\text{ and }\llambda_m\in B' \text{ such that : }\\
& & \llambda_1\lhd_{\theta}\llambda_2\lhd_{\theta'}\llambda'_2\lhd_{\theta}\llambda'_3\lhd_{\theta'}\cdots\lhd_{\theta'}\llambda_m,\\\\

& \Longrightarrow & \exists~ \llambda_1\in B,~ \llambda_m\in B' \text{ such that } \llambda_1\lhd_{\theta'}\llambda_m,\\\\

& \Longrightarrow & B\lhd_{\theta'}B'.

\end{array}$$

\end{proof}

\end{abs}


\section{Different orders on $CM_{\h'}(G(\l,e,n))$.}\label{glen}

Now it is time to try to generalize all these constructions and properties to the whole family of imprimitive reflexion groups.
Let $e\in\Z_{>0}$ be such that $e\mid \l$ and let $p=\frac{\l}{e}$. Let $A(\l,e,n)$ be the group of all diagonal matrices whose diagonal entries are powers of $\zeta_\l=\exp\left(\frac{2\pi\sqrt{-1}}{\l}\right)$ and whose determinant is a $p^\text{th}$ root of unity. The group $\cat{S}_n$, considered as the group of all $n\times n$ permutation matrices, normalizes $A(\l,e,n)$ and the imprimitive group $K=G(\l,e,n)$ is defined to be the semidirect product of $A(\l,e,n)$ by $\cat{S}_n$. Moreover $K$ is a normal subgroup of $W=G(\l,1,n)$ and the quotient group $W/K$ is the cyclic group $C_e$. We denote by $\res$ the restriction functor $\C[W]\rightarrow \C[K]$.

\begin{abs} Let us describe the Calogero-Moser partition of $K$, following \cite{be:th:09}. Let $\h'\in\H$ such that $$H'_{j+p}=H'_j, ~ \forall~ 0\lq j \lq\l-1$$ then $\h'$ belongs to a GIT wall, $H_{\h'}(K)$ is a subalgebra of $H_{\h'}(W)$ and the $CM_{\h'}(K)$ partition is described as follows.

\begin{Th}[Bellamy]
Let $B$ be a block of $CM_{\h'}(W)$.

\begin{enumerate}[(i)]
\item If $\llambda$ is a $e$-stuttering $\l$-multipartition of $n$ such that $B=\{\llambda\}$ then $$\Gamma(B):=\{\mmu \in \irr(K) \text{ occurring as a summand of } \res \llambda, \text{ for }\llambda\in B\}$$ is a disjoint union of $|\widehat{C}_\lambda|$ blocks (where $\widehat{C}_\lambda$ is the stabilizer of $\llambda$ in $\ho(C_e,\C^*)$ with respect to $\llambda$),
\item otherwise $\Gamma(B)$ is a $CM_{\h'}(K)$ block.
\end{enumerate}
\end{Th}
 
 \begin{Rem}
 (i) Let $\llambda=(\lambda^1,\dots,\lambda^\l)\in\mathcal{P}(\l,n)$. We rewrite $\llambda$ as $\llambda=(\llambda_1,\dots,\llambda_e)$ where $\llambda_i=(\lambda^{(i-1)p+1},\dots,\lambda^{ip})$. We say that $\llambda$ is \textit{$e$-stuttering} if $\llambda_i=\llambda_j$ for all $1\lq i,j\lq e$.
 
\nd (ii) This result comes from \cite[4.11]{be:th:09} which can be generalized to the case where $W/K$ is abelian, see \cite{li:al:12}.

\nd (iii) If there exists a $e$-stuttering $\l$-multipartition of $n$ then $e\mid n$. Therefore if $e\nmid n$ then the Calogero-Moser partitions of $G(\l,1,n)$ and $G(\l,e,n)$ are in one-to-one correspondence.
 \end{Rem}
 \end{abs}
 
 \begin{abs}{\bf The $e\nmid n$ case.} \hspace{0,05cm} In order to describe geometrically $CM_{\h'}(G(\l,e,n))$ for  $e\nmid n$, we have to find a variety $\M_{\theta'}(e,n)$ such that : \begin{itemize} 
 \item there exists a $\C^*$-equivariant morphism  $\M_{\theta'}(e,n)\twoheadrightarrow (\cat{h}\times\cat{h}^*)/G(\l,e,n)$,
 \item $(\M_{\theta'}(e,n))^{\C^*}$ is in one-to-one correspondence with $(\M_{\theta'}(n))^{\C^*}$.
 \end{itemize}
 But we have the following diagram : \[ \xymatrix{ \M_{\theta'}(n) \ar@{->>}[rd]_{\pi_{\theta',0}} & & (\cat{h}\times\cat{h}^*)/G(\l,e,n) \ar@{->>}[ld]^{p}\\
 &  (\cat{h}\times\cat{h}^*)/G(\l,1,n) & }  \] where $p$ corresponds to the quotient by the cyclic group $C_e\simeq G(\l,1,n)/G(\l,e,n)$. Thus we can consider the reduced variety associated with the fiber product $\M_{\theta'}(n)\times_{(\cat{h}\times\cat{h}^*)/G(\l,1,n)}(\cat{h}\times\cat{h}^*)/G(\l,e,n)$ : $$\M_{\theta'}(e,n):= \{(m,x)\in\M_{\theta'}(n)\times(\cat{h}\times\cat{h}^*)/G(\l,e,n), \text{ such that } \pi_{\theta',0}(m)=p(x)\}.$$ This variety is the one we were looking for because the only fixed point of $(\cat{h}\times\cat{h}^*)/G(\l,e,n)$ for the hyperbolic $\C^*$-action is $0$. Therefore, if we consider the diagonal $\C^*$-action on $\M_{\theta'}(e,n)$ then $$(\M_{\theta'}(e,n))^{\C^*}\leftrightarrow(\M_{\theta'}(n))^{\C^*}.$$ Moreover, since the following diagram commutes 
 
 \[ \xymatrix{ & \M_{\theta'}(e,n) \ar@{->>}[ld]_{f} \ar@{->>}[rd] & \\
\M_{\theta'}(n) \ar@{->>}[rd]_{\pi_{\theta',0}} & &( \cat{h}\times\cat{h}^*)/G(\l,e,n)\ar@{->>}[ld]^{/C_e}\\
 &  (\cat{h}\times\cat{h}^*)/G(\l,1,n) & }  \]
  we have : $$\M_{\theta'}(n)\simeq\M_{\theta'}(e,n)/C_e,$$ where $C_e$ acts on the second component $x\in(\cat{h}\times\cat{h}^*)/G(\l,e,n)$ of $\M_{\theta'}(e,n)$.
 
 \begin{prop}
 The variety $\M_{\theta'}(e,n)$ is irreducible.
 \end{prop}
 
 \begin{Rem}
 We can construct the variety $\M_{\theta'}(e,n)$ in the same manner for all parameters $e$ and $\ttheta\in\Q^\l$. This variety keeps its properties but its fixed points do not describe $CM_{\h'}(G(\l,e,n))$ in general.
 \end{Rem}

 \begin{proof}
 By construction of $\pi_{\theta',0}$ (see \cite[2.12]{na:qu:09}) there exists two open sets $U$ and $V$ such that $\pi_{\theta',0|U}~:~U\rightarrow V$ is an isomorphism. Therefore if we denote by $f$ the projection $\M_{\theta'}(e,n)\twoheadrightarrow \M_{\theta'}(n)$ then $\overline{f^{-1}(U)}$ is an irreducible component of $\M_{\theta'}(e,n)$ of maximal dimension. Moreover $f$ is a finite morphism thus all the irreducible components of $\M_{\theta'}(e,n)$ are of the form $\eta\cdot \overline{f^{-1}(U)}$ with $\eta \in C_e$. But $f^{-1}(U)$ is stable under the action of $C_e$, thus $\overline{f^{-1}(U)}$ is the unique irreducible component of $\M_{\theta'}(e,n)$.
 \end{proof}
 
 The fact that $\M_{\theta'}(n)\simeq\M_{\theta'}(e,n)/C_e$ and that the action of $C_e$ on $\C^*$-fixed points is trivial imply, according to Corollary \ref{2.2}, that the orderings defined on $(\M_{\theta'}(e,n))^{\C^*}$ and $(\M_{\theta'}(n))^{\C^*}$ by the Bialynicki-Birula decomposition are equivalent. Therefore, if we denote by $\{\Gamma(B),B\in CM_{\h'}(G(\l,1,n))\}$ the partition $CM_{\h'}(G(\l,e,n))$ and if we define \textit{the geometric order on $CM_{\h'}(G(\l,e,n))$} in the following way : for $\Gamma(B)$, $\Gamma(B')\in CM_{\h'}(G(\l,e,n))$ : $$\Gamma(B)\prec_{\theta',e} \Gamma(B')~\Longleftrightarrow~X_{\theta',e}(B)\prec X_{\theta',e}(B'),$$

\nd where $X_{\theta',e}(B)$ is the $\C^*$-fixed point of $\M_{\theta'}(e,n)$ corresponding to $\Gamma(B)$, then we have the following result. 

\begin{Th}\label{ordres-geo-G(l,e,n)} Let $B$, $B'$ be two blocks of $CM_{\h'}(G(\l,1,n))$, then :  $$B\prec_{\theta'} B'~\Longleftrightarrow~\Gamma(B)\prec_{\theta',e}\Gamma(B').$$
\end{Th}

Therefore it is natural to define a \textit{combinatorial ordering on $CM_{\h'}(G(\l,e,n))$} in the following way : for $\Gamma(B)$, $\Gamma(B')\in CM_{\h'}(G(\l,e,n))$ : $$\Gamma(B)\lhd_{\theta',e} \Gamma(B')~\Longleftrightarrow~B\lhd_{\theta'} B'.$$
 Moreover, by construction, the $\fc$-function is constant inside the Calogero-Moser blocks of a complex reflexion group and it is such that : $\fc_{\h'}(B)=\fc_{\h'}(\Gamma(B))$. We can check that this property still holds for the $\a$-function on $\irr G(\l,e,n)$ and $\irr G(\l,1,n)$, according to \cite[Lem. A.7.1 and prop. 2.3.15]{ch:bl:09}. Thanks to all these constructions, it is easy to generalize the results of \ref{order}.
 
 \begin{Th}\label{touslesordresg(l,e,n)}
Let $\Gamma(B)$, $\Gamma(B')$ be two blocks of  $CM_{\h'}(G(\l,e,n))$ then $$\Gamma(B)\prec_{\theta',e} \Gamma(B') ~\Longrightarrow~\left\{\begin{array}{rcl}\a_{\h'}(\Gamma(B))&\lq&\a_{\h'} (\Gamma(B'))\\
\fc_{\h'}(\Gamma(B))&\gq&\fc_{\h'} (\Gamma(B')). \end{array}\right.$$ and for $\l=2$ : $$\Gamma(B)\prec_{\theta',e}\Gamma( B') ~\Longrightarrow~\Gamma(B)\lhd_{\theta',e} \Gamma(B').$$
\end{Th}

 \end{abs}

  \begin{abs}{\bf An example : $\mathbf{G(2,2,2)\lhd G(2,1,2)}$.} \hspace{0,05cm} To deal with the case $e\mid n$ we have to blow up some fixed points of $\M_{\theta'}(2,2)$ (those which correspond to the blocks $\{\llambda\}$, when $\llambda$ is $e$-stuttering). Here we study the case $\l=n=2$ in order to have an idea of the strategy for the general case.
  
  We know that $(\cat{h}\times\cat{h}^*)/G(2,2,2)$ is the product of two cones and that for $\h'=(1,0)$, the group $G(2,1,2)$ has three blocks : $\{(2;\emptyset),(\emptyset;2)\}$, $\{(11;\emptyset),(\emptyset;11)\}$ and $\{(1;1)\}$. Since the multipartition  $\{(1;1)\}$ is $2$-stuttering, $G(2,2,2)$ has four blocks for this parameter. Hence, in order to describe geometrically the blocks of $G(2,2,2)$, we have to construct a variety $\M$ such that : \begin{itemize} 
 \item there exists a $\C^*$-equivariant morphism  $\M \twoheadrightarrow (\cat{h}\times\cat{h}^*)/G(2,2,2)$,
 \item $|(\M)^{\C^*}|=4$.
 \end{itemize} In fact, the crepant resolution : $$\M=\M_{(0,1)}(1)\times\M_{(0,1)}(1)$$ satisfies these properties.
  
\newpage
  
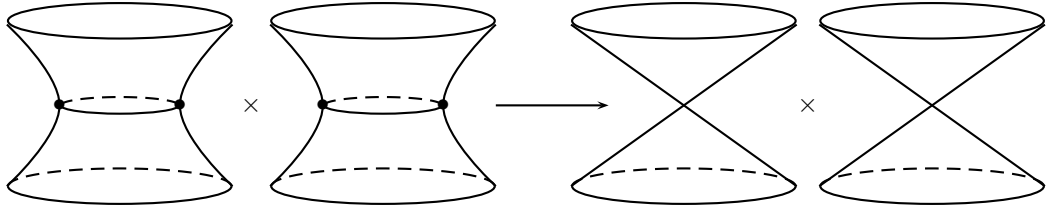
\hspace{1,5cm}  \begin{pspicture}(0,0)(10,2,3) 

\psellipticarc[linestyle=dashed](8.5,0.1)(1.5,0.25){0}{180} 
\psellipticarc(8.5,0.1)(1.5,0.25){180}{360} 
\psellipse(8.5,2.3)(1.5,0.25) 
\psline(7,2.25)(10,0.1) 
\psline(10,2.25)(7,0.1) 

\psellipticarc[linestyle=dashed](11.8,0.1)(1.5,0.25){0}{180} 
\psellipticarc(11.8,0.1)(1.5,0.25){180}{360} 
\psellipse(11.8,2.3)(1.5,0.25) 
\psline(10.3,2.25)(13.3,0.1) 
\psline(13.3,2.25)(10.3,0.1) 

\psellipticarc[linestyle=dashed](1,0.1)(1.5,0.25){0}{180}
\psellipticarc(1,0.1)(1.5,0.25){180}{360}
\psellipse(1,2.3)(1.5,0.25)
\psellipticarc[linestyle=dashed](1,1.175)(0.8,0.125){0}{180}
\psellipticarc(1,1.175)(0.8,0.125){180}{360}

\pscurve(-0.5,2.25)(0.2,1.175)(-0.5,0.1)

\pscurve(2.5,2.25)(1.8,1.175)(2.5,0.1)

\psellipticarc[linestyle=dashed](4.5,0.1)(1.5,0.25){0}{180}
\psellipticarc(4.5,0.1)(1.5,0.25){180}{360}
\psellipse(4.5,2.3)(1.5,0.25)
\pscurve(3,2.25)(3.7,1.175)(3,0.1)
\pscurve(6,2.25)(5.3,1.175)(6,0.1)
\psellipticarc[linestyle=dashed](4.5,1.175)(0.8,0.125){0}{180}
\psellipticarc(4.5,1.175)(0.8,0.125){180}{360}

\psline{->}(6,1.175)(7.5,1.175)

\rput(10.15,1.175){$\times$}
\rput(2.75,1.175){$\times$}

\rput(0.2,1.175){$\bullet$}
\rput(1.8,1.175){$\bullet$}
\rput(3.7,1.175){$\bullet$}
\rput(5.3,1.175){$\bullet$}




\rput(6,-0.8){$\M\twoheadrightarrow (\cat{h}\times\cat{h}^*)/G(2,2,2)$} 
\end{pspicture} 
        
    \vspace{1,5cm}

    Thus, in order to generalize this example, we have to put in relation these two varieties : $\M$ and $\M_{\theta'}(2,2)$. We have the following diagram :  \[ \xymatrix{ & \M_{\theta'}(2,2) \ar[r]^{p_1} \ar[d]_{p_2} & \M_{\theta'}(2) \ar[d]^{\pi_{\theta',0}} \\
    \M \ar[r]     &   (\cat{h}\times\cat{h}^*)/G(2,2,2) \ar[r]_p & (\cat{h}\times\cat{h}^*)/G(2,1,2)  }\] and if we denote by $X$ the regular set of  $(\cat{h}\times\cat{h}^*)/G(2,2,2)$ and if we suppose that $\pi_{\theta',0}$ is an isomorphism above $p(X)$ (this hypothesis is consistent because $\ttheta'$ depends on the variety $ (\cat{h}\times\cat{h}^*)/G(2,2,2)$) then we can show that $\M$ is the normalization of $\M_{\theta'}(2,2)$ (see \cite{li:al:12} for more details).
 The idea of the proof consists on describing these two varieties as blowups of the singular variety $(\cat{h}\times\cat{h}^*)/G(2,2,2)$ in order to relate them. First at all, the variety $\M$ satisfies $$\begin{array}{rcl}\M&=&\{((a,b,c),(a',b',c'),[\alpha:\beta:\gamma:\delta])\in (\cat{h}\times\cat{h}^*)/G(2,2,2)\times\p^3(\C),\\
& & ~~~a\gamma=c\alpha,~a\delta=c\beta,~a\alpha=b\gamma,~a\beta=b\delta,~a'\delta=c'\gamma,~a'\beta=c'\alpha,\\
& & ~~~a'\alpha=b'\beta,a'\gamma=b'\delta,~\alpha\delta=\beta\gamma\}\end{array}$$ and then it is isomorphic to the blowup of $(\cat{h}\times\cat{h}^*)/G(2,2,2)$ with respect to $I$, the maximal ideal which describes the singular set of $(\cat{h}\times\cat{h}^*)/G(2,2,2)=\{(a,b,c,a',b',c')\in \mathbb{A}^6(\C),~a^2=bc,~a'^2=b'c'\}$, i.e. $I=<A,B,C>\cdot<A',B',C'>$. Secondly, under the following hypothesis 

\vspace{-0,5cm}

\begin{center}
            \colorbox{gris25}{
                    \fbox{\begin{minipage}{0,977\textwidth} \textbf{Hypothesis (H).} The morphism $\pi_{\theta',0}~:~\M_{\theta'}(2)~\longrightarrow~(\cat{h}\times\cat{h}^*)/G(2,1,2)$ is an isomorphism above $p(X)$. \end{minipage}
                    }
            }
    \end{center} and thanks to the definition of $\M_{\theta'}(2,2)$, it is easily seen that the second projection $p_2~:~\M_{\theta'}(2,2)~\longrightarrow~(\cat{h}\times\cat{h}^*)/G(2,2,2)$ verifies the same property. Then, according to  \cite[7.17]{ha:ge:93}, $\M_{\theta'}(2,2)$ is the blowup of $(\cat{h}\times\cat{h}^*)/G(2,2,2)$ with respect to $J\subset I$ and therefore there exists a projective, surjective and $\C^*$-equivariant morphism $$\Phi~:~\M~\longrightarrow~\M_{\theta'}(2,2).$$ The cyclic group $C_2$ acts naturally on $\M$ and $\M_{\theta'}(2,2)$ thus the existence of $\Phi$ implies the existence of a projective and $\C^*$-equivariant morphism $$\tilde\Phi~:~\M/C_2\longrightarrow \M_{\theta'}(2,2)/C_2\simeq \M_{\theta'}(2)$$ which is an isomorphism over $\pi_{\theta',0}^{-1}(p(X))$ by construction. Moreover $\tilde\Phi$ is a bijection between the three fixed points of these varieties : indeed for $x\in(\M_{\theta'}(2))^{\C^*}$, since $\tilde{\Phi}^{-1}(x)$ is projective, it contains one fixed point and if $z\in\tilde{\Phi}^{-1}(x)$ is not fixed then $\lim_{\eta\rightarrow 0}\eta\cdot z$ and $\lim_{\eta\rightarrow +\infty}\eta\cdot z$ are two different fixed points on $\tilde{\Phi}^{-1}(x)$ and that is impossible. Finally, let $x\not\in\pi_{\theta',0}^{-1}(p(X))$ be not fixed, since $x\in\pi_{\theta',0}^{-1}\left((\cat{h}\times\{0\})/G(2,1,2)\cup(\{0\}\times\cat{h}^*)/G(2,1,2)\right)$, $\lim_{\eta\rightarrow 0}\eta\cdot z$ or $\lim_{\eta\rightarrow +\infty}\eta\cdot z$ exists. Moreover if $\dim \tilde\Phi^{-1}(x)\gq 1$, according to \cite[Chap AG, 10.3]{bo:li:91}, this limit belongs to the closed set $\mathcal{F}=\{z\in\M_{\theta'}(2), \dim\tilde\Phi^{-1}(z)\gq 1\}$ but, since this limit is fixed, this is impossible. Therefore, for all $x\not\in\pi_{\theta',0}^{-1}(p(X))$, the fiber $\tilde{\Phi}^{-1}(x)$ is finite and we can conclude thanks to the normality of the varieties. Hence, to prove that  $\M$ is the normalization of $\M_{\theta'}(2,2)$ under the hypothesis \textbf{(H)}, we just have to show that $\M$ is the normalization of $(\M/C_2)\times_{(\cat{h}\times\cat{h}^*)/G(2,1,2)}(\cat{h}\times\cat{h}^*)/G(2,2,2)$. It follows from the following diagram : \[ \xymatrix{ & \M \ar[rd]^{\bar p} \ar[ld]_{\Phi} & \\
(\M/C_2)\times_{(\cat{h}\times\cat{h}^*)/G(2,1,2)}(\cat{h}\times\cat{h}^*)/G(2,2,2) \ar[rr]_{~~~~~~~~~~~~~~\bar{p_1}}] \ar[d] _{\bar{p_2}} & & \M/C_2 \ar[d]^{\bar\pi}\\
(\cat{h}\times\cat{h}^*)/G(2,2,2) \ar[rr]_p & & (\cat{h}\times\cat{h}^*)/G(2,1,2)} \] because $\Phi$ is sujective, $\M$ is normal and $p$ and $\overline{p_1}$ have degree $2$.

  For general parameters $\l,~e,~n$ such that $e\mid \l$ and $e\mid n$, the normalization of $\M_{\theta'}(e,n)$ could be the variety whose $\C^*$-fixed points describe $CM_{\h'}(G(\l,e,n))$ because normalizing a variety consists on blowing up some subvarieties thus we can expect that the good points could be blown up, moreover the property of normality is necessary for the definition of the geometric ordering and, as the Calogero-Moser space $\spec Z(H_{\h'}(G(\l,e,n)))$ is normal, this property could be interesting too, in order to compare these two varieties, as in the $G(\l,1,n)$ case.

 \end{abs}
 

\bibliographystyle{alpha}

\end{document}